\documentclass[11pt,reqno]{amsart}

\pdfoutput = 1

\usepackage[english]{babel}
\usepackage{amsmath,amssymb,amsthm}
\usepackage{amsfonts}
\usepackage{amsxtra}

\usepackage{array,dcolumn}
\usepackage{graphicx}
\usepackage[hyperfootnotes=true,
        pdffitwindow=true,
        plainpages=false,
        pdfpagelabels=true,
        pdfpagemode=UseOutlines,
        pdfpagelayout=SinglePage,
                hyperindex,]{hyperref}

\usepackage[T1]{fontenc}
\usepackage[utf8]{inputenc}
\usepackage[activate={true,nocompatibility},spacing,kerning]{microtype}
\usepackage{bm}
\usepackage{bbm}
\usepackage[sc,osf]{mathpazo}
\microtypecontext{spacing=nonfrench}

 \calclayout

\newcommand{\mathsym}[1]{{}}
\newcommand{\unicode}[1]{{}}

\theoremstyle{plain}
\newtheorem{theorem}{Theorem}
\newtheorem{lemma}[theorem]{Lemma}
\newtheorem{corollary}[theorem]{Corollary}
\newtheorem{proposition}[theorem]{Proposition}

\theoremstyle{definition}

\theoremstyle{remark}
\newtheorem{remark}[theorem]{Remark}

\newcommand{\Z}{\mathbb Z}
\newcommand{\R}{\mathbb R}
\newcommand{\C}{\mathbb C}
\newcommand{\+}{\!+\!}

\newcommand{\half}{\tfrac{1}{2}}

\renewcommand{\leq}{\leqslant}
\renewcommand{\geq}{\geqslant}

\numberwithin{equation}{section}

\begin{document}


\title{Volumes and distributions for random unimodular complex and quaternion lattices}
\author{Peter J. Forrester}
\address{Department of Mathematics and Statistics, 
ARC Centre of Excellence for Mathematical \& Statistical Frontiers,
University of Melbourne, Victoria 3010, Australia}
\email{pjforr@unimelb.edu.au}
\author{Jiyuan Zhang}
\address{Department of Mathematics and Statistics, 
University of Melbourne, Victoria 3010, Australia}
\email{jiyuanz@student.unimelb.edu.au}
\date{\today}


\begin{abstract}
Two themes associated with invariant measures on the matrix groups ${\rm SL}_N(\mathbb F)$, with $\mathbb F = \mathbb R, \mathbb C$ or $\mathbb H$, and their corresponding lattices parametrised by ${\rm SL}_N(\mathbb F)/{\rm SL}_N(\mathbb O)$, $\mathbb O$ being an appropriate Euclidean ring of integers, are considered. 
The first is the computation of the volume of the subset of ${\rm SL}_N(\mathbb F)$ with bounded 2-norm or Frobenius norm. Key here is the decomposition of measure in terms of the singular values. The form of the volume, for large values of the bound, is relevant 
to asymptotic counting problems in ${\rm SL}_N(\mathcal O)$. The second is the problem of lattice reduction in the case $N=2$. A unified proof of 
the validity of the appropriate analogue of the Lagrange--Gauss algorithm for computing the shortest basis is given. A decomposition of measure corresponding
to the QR decomposition is used to specify the invariant measure in the coordinates of the shortest basis vectors. With $\mathbb F = \mathbb C$ this
allows for the exact computation of the PDF of  the first minimum (for $\mathcal O = \mathbb Z[i]$ and $\mathbb Z[(1+\sqrt{-3})/2]$), and the PDF of the
second minimum and that of the angle between the minimal basis vectors (for $\mathcal O = \mathbb Z[i]$).
It also encodes the specification of fundamental domains of the corresponding quotient spaces. Integration over the latter gives rise to certain number theoretic constants, which are also present in the  asymptotic forms of the PDFs of the lengths of the shortest basis vectors.
Siegel's mean value 
gives an alternative method to compute the arithmetic constants, allowing in particular the computation of the leading form
of the PDF of the first minimum for $\mathbb F = \mathbb H$ and $\mathcal O$ the Hurwitz integers, for which direct integration was not possible.
\end{abstract}


\maketitle


\section{Introduction}\label{s1}
Let $\mathcal{B}=\{\mathbf{b}_0, \mathbf{b}_1,\ldots, \mathbf{b}_{d-1}\}$ be a basis of $\R^d$, and require that the corresponding parallelotope have unit volume. Let
\begin{equation}\label{1}
\mathcal{L}=\{m_0\mathbf{b}_0+\cdots+m_{d-1}\mathbf{b}_{d-1}\,|\,m_0,\ldots,m_{d-1}\in\Z\}
\end{equation}
denote the corresponding lattice. The Minkowski-Hlawka theorem tells us that for large $d$, there exists lattices such that the shortest vectors have length proportional to $\sqrt{d}$. By the Minkowski convex body theorem this is also the maximum possible order of magnitude of the shortest vectors; see e.g. \cite{CS99}. Siegel \cite{Si45} introduced the notion of a random lattice, and was able to show that for large dimension $d$, a random lattice will typically achieve the Minkowski-Hlawka bound.

The construction of Siegel of a random lattice requires first the specification of the unique invariant measure for the matrix group $\mathrm{SL}_N(\R)$; each such matrix is interpreted as having columns forming a basis $\mathcal{B}$. One also requires the fact that the quotient space $\mathrm{SL}_N(\R)/\mathrm{SL}_N(\Z)$ can be identified with the set of lattices, and that this quotient space has finite volume with respect to the invariant measure.

In a recent work \cite{Fo16} by one of the present authors, a viewpoint from random matrix theory was taken on the computation of volumes associated with $\mathrm{SL}_N(\R)$, and this led to a Monte Carlo procedure to generate random lattices in the sense of Siegel. In low dimensions $d=2$, $3$ and $4$ there are fast exact lattice reduction algorithms to find the shortest lattice vectors \cite{Se01,NS04} -- the case $d=2$ is classical being due to Lagrange and Gauss; see e.g. \cite{Br12}. These were implemented in dimensions two and three to obtain histograms of the lengths and their mutual angles; in dimension two the exact functional forms were obtained by integration over the fundamental domain. For general $d$, it was shown how a mean value theorem derived by Siegel in \cite{Si45} implies the exact functional form of the distribution $P_{\mathrm{short}}(t)$ of the length of the shortest vector for general $d$,
\begin{equation}\label{2}
P_{\mathrm{short}}(t)\underset{s\rightarrow0}{\sim}\frac{dv_d}{2\zeta(d)}t^{d-1},
\end{equation}
where $\zeta(x)$ denotes the Riemann zeta function, and $v_d$ the volume of the unit ball in dimension $d$ (actually only the case $d=3$ was presented, but the derivation applies for general $d$ to give (\ref{2})).

In random matrix theory, matrix groups with entries from any of the three associative normed division algebras $\R$, $\C$ or $\mathbb{H}$ are fundamental \cite{Dy62c} (dropping the requirement of associativity permits the octonions $\mathbb{O}$ to be added to the list; see the recent work \cite{Fo16a} for spectral properties of various ensembles of $2\times2$ and $3\times3$ Hermitian matrices with entries in $\mathbb{O}$). As such, attention is drawn to extending the considerations of \cite{Fo16} to the case of complex and quaternion vector spaces $\C^n$ and $\mathbb{H}^n$. One remarks that lattices in these vector spaces, with scalars equal to the Gaussian integers and Eisenstein integers for $\mathbb C^2$, and Hurwitz integers for $\mathbb H^n$, received earlier attention for their application to signal processing in wireless communication \cite{YW02, GLM09, WSJM11,SYHS13}, and their consequences for lattice packing bounds \cite{Va11} respectively. The study \cite{Na96} extends the LLL   lattice
reduction algorithm to these settings.

Of particular interest from the viewpoint of \cite{Fo16} are the invariant measure for $\mathrm{SL}_N(\C)$ and $\mathrm{SL}_N(\mathbb{H})$, the associated volumes, and the corresponding lattice reduction problems. Following the work of Jack and Macbeath in the case of $\mathrm{SL}_N(\R)$, we begin in \S2 by using the singular value factorisation to decompose the invariant measures. To obtain a finite volume, a certain truncation must be introduced, most naturally by restricting the norm $\lVert M\rVert$ to be bounded by a value $R$. We do this in the case of the 2-norm $\lVert M\rVert_{2}:=\mu_1$, where $\mu_1$ is the largest singular value of $M$, and the Frobenius norm  $\lVert M\rVert_F:=\left(\sum_{j=1}^N\mu_j^2\right)^{1/2}$, where $\mu_j$ is the $j$\textsuperscript{th} largest singular value. The large $R$ form of the volume is of particular relevance due to counting formulas of the type \cite{DRS93}
\begin{equation}\label{3}
\#\{\gamma\,:\,\gamma\in\mathrm{SL}_N(\Z),\,\lVert\gamma\rVert\leq R\}\underset{R\rightarrow\infty}{\sim}\frac{1}{\mathrm{vol}\,\Gamma}\int_{\lVert G\rVert\leq R}(\mathrm{d}G).
\end{equation}
Here $(\mathrm{d}G)$ is the Haar measure on $\mathrm{SL}_N(\R)$, and vol $\Gamma$ the volume of the corresponding fundamental domain. 
A generalisation of (\ref{3}) applying to lattice subgroups of topological groups, and in particular
\begin{equation}\label{3a}
\# \{ \gamma: \gamma \in {\rm SL}_N(\mathbb Z[i]) ,\lVert\gamma\rVert\leq R\},
\end{equation}
is given in \cite[Th.~1.5]{GN10}, and has the same structure as (\ref{3}). As an application of our evaluation of the volume of
a ball in SL${}_N(\mathbb C)$ we are able to compute the leading large $R$ form of (\ref{3a}), up to the value of vol $\Gamma$; in the case $N=2$ this can be determined and we obtain the explicit
asymptotic expression (\ref{RF3}) below.


 For lattices in $\C^2$ with scalars from particular rings of complex quadratic  integers, there is a generalisation of the Lagrange-Gauss algorithm that allows for the determination of a reduced basis $\{\bm{\alpha}, \bm{\beta}\}$ with the shortest possible lengths.
For the Gaussian and Eisenstein integers this has been noted previously \cite{YW02,SYHS13}, although our proofs
given in \S 4.1 are different and apply to all cases at once.  They are motivated by known theory in the real case, which we revise
in \S3. Another point covered in \S3 is the observation in \cite{DFV97}
that the original Lagrange-Gauss algorithm is equivalent to a simple mapping in the complex plane, related to the Gauss map for continued fractions. We show in \S 4.2 that in the case of lattices in $\C^2$, the generalisation of the Lagrange-Gauss algorithm for lattice reduction can be written as a scalar  mappings of quaternions.

In the Gaussian case, 
the PDF for the lengths of the reduced basis vectors and the scaled inner product $\left|\,\overline{\bm{\alpha}}\cdot\bm{\beta}/\lVert\bm{\alpha}\rVert\lVert\bm{\beta}\rVert\,\right|$ are computed analytically in Section 4.4.
For values of $s$ less than $1$, it is found $P_{\mathrm{short}}(s)=cs^3$ for a particular $c$, thus relating to \eqref{2} with $d=4$. This latter result is found too in the case of the Eisenstein integers, for a different value of $c$,
upon the exact calculation of the functional form of the PDF of the length of the shortest vector 
carried out in Section 4.5.
Siegel's mean value
theorem \cite{Si45} is used to give an independent computation of $c$ in the two cases.

Analogous considerations are applied to lattices formed from vectors in $\mathbb{H}^2$ with scalars the integer Hurwitz quaternions in Section 5; now $P_{\mathrm{short}}(s)\underset{s\rightarrow0}{\sim}ks^7$ for a particular $k$, thus relating to \eqref{2} with $d=8$. Here the direct computation of $k$ as done for the case of the Gaussian and Eisenstein integers
appears not to be tractable, but the exact value can be found indirectly by use of Siegel's mean value
theorem.


\section{Invariant measure and volumes for $\mathrm{SL}_N(\C)$ and $\mathrm{SL}_N(\mathbb{H})$}\label{S2}
\subsection{Invariant measure}\label{s2.1}
By way of preliminaries, one recalls that the quaternions $\mathbb{H}$ are a non-commutative algebra with elements of the form
\begin{equation}\label{aa}
a_0+a_1\mathrm{i}+a_2\mathrm{j}+a_3\mathrm{k},
\end{equation}
where $a_0,\ldots,a_3\in\R$, $\mathrm{i}^2=\mathrm{j}^2=\mathrm{k}^2=-1$, $\mathrm{ijk}=-1$, and each distinct pair of $\{\mathrm{i},\mathrm{j},\mathrm{k}\}$ anti-commutes. However, matrix groups with elements from $\mathbb{H}$ typically make use of the representation of quaternions as $2\times2$ complex matrices
\begin{equation}\label{6}
\begin{bmatrix}z&w\\-\overline{w}&\overline{z}\end{bmatrix},\quad z=a_0+a_i\mathrm{i},\,w=a_2+a_3\mathrm{i}.
\end{equation}
Thus for example matrices from $\mathrm{GL}_N(\mathbb{H})$ and $\mathrm{SL}_N(\mathbb{H})$ are then $N\times N$ block matrices with each entry a $2\times2$ block of the form \eqref{6}, and hence $2N\times 2N$ complex matrices.

Let $G\in\mathrm{GL}_N(\mathbb{F})$, where $\mathbb{F}=\R$, $\C$ or $\mathbb{H}$. Label by $\beta=1$, $2$, $4$ respectively according to the number of independent real parts in an element of $\mathbb{F}$. The symbol $(\mathrm{d}G)$ denotes the product of differentials of all the real and imaginary parts of $G$. Since for fixed $A\in\mathrm{GL}_N(\mathbb{F})$
\begin{equation*}
(\mathrm{d}AG)=(\mathrm{d}GA)=|\mathrm{det}\,A|^{\beta N}(\mathrm{d}G)
\end{equation*}
(these follow from e.g. \cite[Prop. 3.2.4]{Fo10}), one has that
\begin{equation}\label{7}
\frac{(\mathrm{d}G)}{|\mathrm{det}\,G|^{\beta N}}
\end{equation}
is unchanged by both left and right group multiplication, and is thus the left and right invariant Haar measure for the group. In the case of $\mathrm{GL}_N(\R)$ and thus $\beta=1$ \eqref{7} was identified by Siegel \cite{Si45}. Matrices in $\mathrm{SL}_N(\mathbb{F})$ form the subgroup of $\mathrm{GL}_N(\mathbb{F})$ with unit determinant. Using a delta function distribution to implement this constraint, \eqref{7} becomes
\begin{equation}\label{8}
\delta(1-\mathrm{det}\,G)\,(\mathrm{d}G).
\end{equation}

In preparation for computing volumes associated with \eqref{8}, as done in the pioneering work of Jack and Macbeath \cite{JM59} in the case $\mathbb{F}=\R$, we make use of a singular value decomposition
\begin{equation}\label{9}
G=U^{(\beta)}\mathrm{diag}(\sigma_1,\ldots,\sigma_N)V^{(\beta)},
\end{equation}
where $U^{(\beta)},V^{(\beta)}\in\mathrm{U}_N(\mathbb{F})$ -- the set of $N\times N$ unitary matrices with entries in $\mathbb{F}$. In the case $\beta=4$ each entry in $\mathrm{diag}(\sigma_1,\ldots,\sigma_N)$ is a $2\times2$ block matrix, so viewed as a $2N\times2N$ matrix each $\sigma_i$ is repeated twice along the diagonal. For \eqref{9} to be one-to-one it is required that the singular values be ordered
\begin{equation*}
\sigma_1\geq\sigma_2\geq\cdots\geq\sigma_N \geq 0
\end{equation*}
and that the entries in the first row of $V^{(\beta)}$ be real and positive.

Changing variables according to \eqref{9} gives (see e.g. \cite[Prop. 2]{DG11})
\begin{multline}\label{10}
(\mathrm{d}G)=\left(\frac{2\pi^{\beta/2}}{\Gamma(\beta/2)}\right)^{-N}\left({U^{(\beta)}}^{\dagger}\mathrm{d}U^{(\beta)}\right)\,\left({V^{(\beta)}}^{\dagger}\mathrm{d}V^{(\beta)}\right)
\\ \times\prod_{l=1}^N\sigma_l^{\beta-1}\prod_{1\leq j<k\leq N}(\sigma_j^2-\sigma_k^2)^{\beta}\,\mathrm{d}\sigma_1\cdots\mathrm{d}\sigma_N,
\end{multline}
where $\left({U^{(\beta)}}^{\dagger}\mathrm{d}U^{(\beta)}\right)$ and $\left({V^{(\beta)}}^{\dagger}\mathrm{d}V^{(\beta)}\right)$ are the invariant measure on $\mathrm{U}_N(\mathbb{F})$. For $\mathbb{F}=\R$ and $\C$ this was first identified by Hurwitz \cite{Hu97}; the extension of Hurwitz's ideas to the case of unitary matrices with quaternion entries is given in \cite{DF17}. The factor $\left(\frac{2\pi^{\beta/2}}{\Gamma(\beta/2)}\right)^{-N}$ comes about due to the restriction on the entries in the first row of $V^{(\beta)}$.

Let us now first restrict the matrices $G\in\mathrm{GL}_N(\mathbb{F})$ to have positive determinant, then to have determinant unity by imposing the delta function constraint in \eqref{8}. This requires that we multiply \eqref{10} by
\begin{equation}\label{10a}
\left(\frac{2\pi^{\beta/2}}{\Gamma(\beta/2)}\right)^{-1}\delta\left(1-\prod_{l=1}^N\sigma_l\right),
\end{equation}
where the first factor corresponds to the reduction in volume due to the restriction to positive determinant.
Consequently, with
\begin{equation}\label{11}
D_R^{\lVert\,\cdot\,\rVert_{2}}\left(\mathrm{SL}_N(\mathbb{F})\right)=\{M\in\mathrm{SL}_N(\mathbb{F})\,:\,\sigma_1\leq R\}
\end{equation}
it follows from this modification of \eqref{10} that
\begin{multline}\label{12}
\mathrm{vol}\left(D_R^{\lVert\,\cdot\,\rVert_{2}}\left(\mathrm{SL}_N(\mathbb{F})\right)\right)=\left(\frac{2\pi^{\beta/2}}{\Gamma(\beta/2)}\right)^{-(N+1)}\left(\mathrm{vol}\,\mathrm{U}_N(\mathbb{F})\right)^2
\\ \times\int_{0<\sigma_N<\cdots<\sigma_1<R}\delta(1-\sigma_1\cdots\sigma_N)\prod_{l=1}^N\sigma_l^{\beta-1}\prod_{1\leq j<k\leq N}(\sigma_j^2-\sigma_k^2)^{\beta}\,\mathrm{d}\sigma_1\cdots\mathrm{d}\sigma_N.
\end{multline}
The precise value of $\mathrm{vol}\,\mathrm{U}_N(\mathbb{F})$ depends on the convention used to relate the line element corresponding to the differential ${U^{(\beta)}}^{\dagger}\mathrm{d}U^{(\beta)}$ to the Euclidean line element; see \cite[Remark 2.3]{Fo16}. This convention can be uniquely specified by integrating \eqref{10} against Gaussian weighted matrices $G$ -- see \cite[Remark 2.3]{Fo16} -- with the result \cite[Eq. (1) with $m=n$]{DG11}
\begin{equation}\label{13}
\mathrm{vol}\,\mathrm{U}_N(\mathbb{F})=2^N\prod_{k=1}^N\frac{\pi^{\beta k/2}}{\Gamma(\beta k/2)}.
\end{equation}

In the case $\beta=1$ the multiple integral in \eqref{12} was first evaluated by Jack and Macbeath \cite{JM59}. In the recent work \cite{Fo16} a simplified derivation was given by making use of the Selberg integral \cite{Se44,FW07p,Fo10}. This strategy can be extended to general $\beta$.

\begin{proposition}\label{prop1}
Define
\begin{equation}\label{14}
J_N^{(\beta)}(R):=\int_{R>\sigma_1>\cdots>\sigma_N>0}\delta\left(1-\prod_{l=1}^N\sigma_l\right)\prod_{l=1}^N\sigma_l^{\beta-1}\prod_{1\leq j<k\leq N}(\sigma_j^2-\sigma_k^2)^{\beta}\,\mathrm{d}\sigma_1\cdots\mathrm{d}\sigma_N
\end{equation}
and set
\begin{equation}\label{15}
A_N^{(\beta)}(R)=\frac{2^{-N}}{N!}R^{N(\beta-1)+\beta N(N-1)}\prod_{j=0}^{N-1}\frac{\Gamma(1+j\beta/2)\Gamma\left(1+(j+1)\beta/2\right)}{\Gamma(1+\beta/2)}.
\end{equation}
For $c>0$ we have
\begin{equation}\label{16}
J_N^{(\beta)}(R)=\frac{A_N^{(\beta)}(R)}{2\pi\mathrm{i}}\int_{c-\mathrm{i}\infty}^{c+\mathrm{i}\infty}R^{Ns}\prod_{j=0}^{N-1}\frac{\Gamma\left(\left(s-1+(j+1)\beta\right)/2\right)}{\Gamma\left(\left(s+1+(N+j)\beta\right)/2\right)}\,\mathrm{d}s.
\end{equation}
\end{proposition}
\begin{proof}
Replace the delta function factor $\delta\left(1-\prod_{j=1}^N\sigma_l\right)$ by $\delta\left(t-\prod_{l=1}^N\sigma_l\right)$ and denote \eqref{14} in this setting by $J_N^{(\beta)}(R;t)$. Making the change of variables $\sigma_l^2=x_l$ and taking the Mellin transform of both sides shows
\begin{align*}
\int_0^{\infty}J_N^{(\beta)}(R;t)t^{s-1}\,\mathrm{d}t&=\frac{2^{-N}}{N!}\int_0^{R^2}\mathrm{d}x_1\cdots\int_0^{R^2}\mathrm{d}x_N\,\prod_{l=1}^Nx_l^{(s+\beta)/2-3/2}\prod_{1\leq j<k\leq N}|x_k-x_j|^{\beta}
\\&=\frac{2^{-N}}{N!}R^{N(s+\beta)}R^{\beta N(N-1)-N}S_N\left((s+\beta-3)/2,0,\beta/2\right).
\end{align*}
Here $S_N(a,b,c)$ is the Selberg integral in the notation of \cite[Ch. 4]{Fo10}. Making use of the gamma function evaluation of the Selberg integral \cite{Se44}, \cite[Eq. (4.3)]{Fo10}, and the notation \eqref{15} reduces this to
\begin{equation*}
A_N^{(\beta)}(R)R^{Ns}\prod_{j=0}^{N-1}\frac{\Gamma\left(\left(s-1+(j+1)\beta\right)/2\right)}{\Gamma\left(\left(s+1+(N+j)\beta\right)/2\right)}.
\end{equation*}
As a function of $s$, this is analytic in the right half plane, and uniformly bounded. The standard formula for the inverse Mellin transform can therefore be applied, giving \eqref{16}.
\end{proof}

\begin{remark}\label{remark2}
For future reference we note from \eqref{16}, as an application of the residue theorem, or alternatively by direct computation from \eqref{14}, that for $N=2$
\begin{align}
J_N^{(2)}(R)&=R^4-\frac{1}{R^4}-8\log\,R,\label{17}
\\ J_N^{(4)}(R)&=\frac{R^8}{8}-R^4+\frac{1}{R^4}-\frac{1}{8R^8}+6\log\,R.\label{18}
\end{align}
Consideration of the direct computation of (\ref{14}) shows that for general $N$ and $\beta = 1,2$ or 4, the function $J_N^{(\beta)}(R)$ is a finite series in power functions and logarithms of $R$, which vanishes when $R=1$.
\end{remark}

\begin{remark}\label{remark2a}
The delta function constraint in (\ref{14}) implies that the factor $\prod_{l=1}^N
\sigma_l^{\beta - 1}$ can be replaced by $\prod_{l=1}^N \sigma_l^{\mu}$
for any $\mu > -1$. The independence of $\mu$ manifests itself in (\ref{16}) by $c>0$ being arbitrary.
\end{remark}

\begin{corollary}\label{corollary3}
As $R\rightarrow\infty$, for $(N,\beta) \ne (2,2)$,
\begin{equation}\label{19}
J_N^{(\beta)}(R)=C_{N,\beta}R^{\beta N(N-1)}+{\rm O}\left( \left \{
\begin{array}{ll} R^{\beta N (N-2)}, & \beta = 1,2 \\
R^{\beta N (N-3/2)}, & \beta = 4 \end{array} \right. \right)
\end{equation}
where
\begin{equation}\label{20}
C_{N,\beta}=\frac{2\beta^N}{2^{2N}\Gamma(N\beta/2)}\prod_{j=0}^{N-1}\frac{\Gamma(1+j\beta/2)\Gamma^2\left((j+1)\beta/2\right)}{\Gamma(1+\beta/2)\Gamma\left(1+(N+j-1)\beta/2\right)}
\end{equation}
and
\begin{multline}\label{21}
\mathrm{vol}\left(D_R^{\lVert\,\cdot\,\rVert_{2}}\left(\mathrm{SL}_N(\mathbb{F})\right)\right)=\frac{\pi^{\beta N^2/2}\Gamma(\beta/2)}{\Gamma(N\beta/2)\pi^{\beta/2}}\prod_{j=0}^{N-1}\frac{\Gamma(1+j\beta/2)}{\Gamma\left(1+(N+j-1)\beta/2\right)}R^{\beta N(N-1)}
\\+ {\rm O}\left( \left \{
\begin{array}{ll} R^{\beta N (N-2)}, & \beta = 1,2 \\
R^{\beta N (N-3/2)}, & \beta = 4 \end{array} \right. \right).
\end{multline}
In the case  $(N,\beta) = (2,2)$, the bound on the correction term is ${\rm O}(\log R)$.
\end{corollary}
\begin{proof}
Standard estimates of the gamma function imply that the integrand decays fast enough in the left
half plane that the contour can be closed in the region without changing its value, by Cauchy's theorem.
This allows the the integral to be computed in terms of a sum over its residues. The poles of the integrand occur at $s=1-(j+1)\beta$ $(j=0,\dots,N-1)$ in the cases $\beta = 1,2$; for $\beta = 4$ there are a further set of poles at $s=1-(j+3/2)\beta$ $(j=0,\dots,N-1)$. The leading contribution to the large $R$ expansion results from pole closest to the origin. This occurs at $s=1-\beta$. Evaluating the residue at this point gives  \eqref{19} and \eqref{20}. The residue of the pole second closest to the origin gives the next term in the large $R$ expansion; the order of this term is also a bound since the number of residues is finite. Note that the case $(N,\beta) = (2,2)$because the pole at $s=1-\beta$ goes
from being first to second order..
\end{proof}

Also of interest is the analogue of \eqref{11} for the Frobenius-norm
\begin{equation*}
D_R^{\lVert\,\cdot\,\rVert_F}\left(\mathrm{SL}_N(\mathbb{F})\right)=\left\{M\in\mathrm{SL}_N(\mathbb{F})\,:\,\sum_{j=1}^N\sigma_j^2\leq R^2\right\},
\end{equation*}
for which the analogue of \eqref{12} reads
\begin{equation}\label{22}
\mathrm{vol}\left(D_R^{\lVert\,\cdot\,\rVert_F}\left(\mathrm{SL}_N(\mathbb{F})\right)\right)=\left(\frac{2\pi^{\beta/2}}{\Gamma(\beta/2)}\right)^{-(N+1)}\left(\mathrm{vol}\,\mathrm{U}_N(\mathbb{F})\right)^2\,\widehat{I}_N^{(\beta)}(R),
\end{equation}
where
\begin{equation}\label{23}
\widehat{I}_N^{(\beta)}(R)=\frac{1}{N!}\int_{\sigma_l>0\,:\,\sum_{j=1}^N\sigma_j^2\leq R^2}\delta\left(1-\prod_{l=1}^N\sigma_l\right)\prod_{1\leq j<k\leq N}|\sigma_j^2-\sigma_k^2|^{\beta}\,\mathrm{d}\sigma_1\cdots\mathrm{d}\sigma_N.
\end{equation}
The integral $\widehat{I}_N^{(\beta)}(R)$ was evaluated in \cite[Prop. 2.9]{Fo10} for $\beta=1$, according to a strategy that extends to general $\beta$.

\begin{proposition}\label{prop4}
For $c>0$ we have
\begin{equation}\label{24}
\widehat{I}_N^{(\beta)}(R)=\frac{R^{\beta N(N-1)}}{2^NN!}\prod_{j=1}^N\frac{\Gamma(1+\beta j/2)}{\Gamma(1+\beta/2)}\frac{1}{2\pi\mathrm{i}}\int_{c-\mathrm{i}\infty}^{c+\mathrm{i}\infty}\frac{\prod_{j=1}^N\Gamma\left(s/2+\beta(N-j)/2\right)}{\Gamma\left(sN/2+\beta N(N-1)/2+1\right)}R^{sN}\,\mathrm{d}s.
\end{equation}
\end{proposition}
\begin{proof}
First introduce
$$
K_N^{(\beta)}(r,t)=\frac{1}{N!}\int_0^{\infty}\mathrm{d}\sigma_1\cdots\int_0^{\infty}\mathrm{d}\sigma_N\,\delta\left(r^2-\sum_{p=1}^N\sigma_p^2\right)
\delta\left(t-\prod_{l=1}^N\sigma_l\right)\prod_{1\leq j<k\leq N}|\sigma_j^2-\sigma_k^2|^{\beta}
$$
so that
\begin{equation}\label{25}
\widehat{I}_N^{(\beta)}(R)=2\int_0^R\left.K_N^{(\beta)}(r,t)\right|_{t=1}r\,\mathrm{d}r.
\end{equation}
Forming the Mellin transform with respect to $t$ shows, after minor manipulation including the
change of variables $\sigma_l^2 = x_l$, that
\begin{multline*}
\int_0^{\infty}K_N^{(\beta)}(r,t)t^{s-1}\,\mathrm{d}t
\\=\frac{r^{\beta N(N-1)+sN-2}}{2^NN!}\int_{\R_+^N}\delta\left(1-\sum_{p=1}^Nx_p\right)\prod_{l=1}^Nx_l^{s/2-1}\prod_{1\leq j<k\leq N}|x_k-x_j|^{\beta}\,\mathrm{d}x_1\cdots\mathrm{d}x_N.
\end{multline*}

The multidimensional integral in this expression is closely related to the Selberg integral, and has the known evaluation in terms of gamma functions \cite{ZS01}, \cite[Eq. (4.154)]{Fo10}. 
Substituting this, then integrating both sides over $r \in (0,R)$ shows
\begin{multline*}
\int_0^{\infty}\left(2\int_0^RK_N^{(\beta)}(r,t)r\,\mathrm{dr}\right)t^{s-1}\,\mathrm{d}t
\\=\frac{R^{sN+\beta N(N-1)}}{2^NN!\Gamma\left(sN/2+\beta N(N-1)/2+1\right)}\prod_{j=1}^N\frac{\Gamma\left(s/2+\beta(N-j)/2\right)\Gamma(1+\beta j/2)}{\Gamma(1+\beta/2)}.
\end{multline*}
The stated result \eqref{24} now follows by taking the inverse Mellin transform and setting $t=1$.
\end{proof}

\begin{corollary}\label{corollary5}
As $R\rightarrow\infty$
\begin{equation}\label{26}
\widehat{I}_N^{(\beta)}(R)=\widehat{C}_N^{(\beta)}R^{\beta N(N-1)}+{\rm O}\left ( \left \{
\begin{array}{ll} R^{N(N-2)}, & \beta = 1 \\
R^{2N(N-2)} \log R, & \beta = 2 \\
R^{4N(N-1)-2N}, & \beta = 4,
\end{array} \right. \right ),
\end{equation}
where
\begin{equation}\label{27}
\widehat{C}_N^{(\beta)}(R)=\frac{2}{2^N\Gamma\left(\beta N(N-1)/2+1\right)}\frac{1}{\Gamma(\beta N/2)}\prod_{j=1}^N\frac{\Gamma^2(\beta j/2)}{\Gamma(\beta/2)},
\end{equation}
and
\begin{multline}\label{28}
\mathrm{vol}\,D_R^{\lVert\,\cdot\,\rVert_F}\left(\mathrm{SL}_N(\mathbb{F})\right)=\frac{\pi^{\beta(N^2-1)/2}\Gamma(\beta/2)}{\Gamma(\beta N/2)\Gamma\left(\beta N(N-1)/2+1\right)}R^{\beta N(N-1)} \\ +
{\rm O}\left ( \left \{
\begin{array}{ll} R^{N(N-2)}, & \beta = 1 \\
R^{2N(N-2)} \log R, & \beta = 2 \\
R^{4N(N-1)-2N}, & \beta = 4.
\end{array} \right. \right ).
\end{multline}
\end{corollary}
\begin{proof}
We proceed in an analogous way to the proof of Corollary 3, and begin by shifting the contour to the line parallel to the imaginary axis with $c = -c_\beta-\epsilon$,
$\epsilon > 0$ with $c_\beta = 1$ for $\beta = 1$ and $c_\beta = 2$ for $\beta = 2$ and 4.  
According to the residue theorem, this changes the value of the integral by $2 \pi i$ times the sum of the residue at $s=0$ and
$s=- c_\beta$. The residue at $s=0$ gives the leading terms, and that at $s=- c_\beta$ the leading correction. 
The large $R$ form of the integrand along the shifted contour shows that the order of this leading correction is a bound on the error term.
This establishes
\eqref{26}; \eqref{28} then follows from \eqref{22}.
\end{proof}

\begin{remark}\label{remark6}
The leading terms in \eqref{21} and \eqref{28} are equal for $N=2$, giving in the
case $\beta = 2$ for example
\begin{equation}\label{RF2}
\mathrm{vol}\,D_R^{\lVert\,\cdot\,\rVert}\left(\mathrm{SL}_2(\mathbb{F})\right) =
{\pi^3 \over 2} R^4 + {\rm O}(\log R),
\end{equation}
but for $N>2$ \eqref{28} is smaller, in keeping with the truncation of the integration domain in going from \eqref{14} to \eqref{23}.
\end{remark}

As commented in the Introduction, one interest in the asymptotic volume formulas
(\ref{21}) and (\ref{26}) lies in asymptotic counting formulas of the type (\ref{3}). For example, as a natural extension of
(\ref{3}), one might expect\footnote{F.~Calegari (private correspondence) remarks that in the context
of \cite{DRS93}, or also Eskin--McMullen, \cite[Theorem 1.4]{EM93},
the basic point is that the $\mathbb Z[i]$ points of a semi-simple group $G$ (like SL${}_n$) are the $\mathbb Z$ points of another group $G' = {\rm Res}_{\mathbb Q(i)/\mathbb Q}(G)$ (the Weil restriction of scalars), so one can apply these theorems to $G'$ to show that counting problem in $G$ in the ring of integers of some (any) number field reduces to a volume calculation.}
that
\begin{equation}\label{28a}
\#\{\gamma\,:\,\gamma\in\mathrm{SL}_N(\Z[i]),\,\lVert\gamma\rVert\leq R\}\underset{R\rightarrow\infty}{\sim}\frac{1}{\mathrm{vol}\,  ( \mathrm{SL}_N(\mathbb C) /  \mathrm{SL}_N(\mathbb \Z[i])}
\int_{G \in \mathrm{SL}_N(\mathbb C): \lVert G\rVert\leq R}(\mathrm{d}G),
\end{equation}
where $\mathbb Z[i]$ denotes the Gaussian integers. In fact a general asymptotic counting theorem for lattice subgroups of topological groups,
implying (\ref{28a}), can be found in \cite[Th.~1.5]{GN10}, as cited in the recent work \cite{EB17}.
The leading asymptotics of the integral over $G$ is given by (\ref{21}) with $\beta = 2$ for
$ \lVert \cdot \rVert = \lVert \cdot \rVert_{\rm Op}$ and by (\ref{26}) with $\beta = 2$ for 
$ \lVert \cdot \rVert = \lVert \cdot \rVert_{F}$. 

It remains then to compute $\mathrm{vol}\,  ( \mathrm{SL}_N(\mathbb C) /  \mathrm{SL}_N(\mathbb \Z[i])$
in the same normalisation as that used to compute $\int_{G \in \mathrm{SL}_N(\mathbb C): \lVert G\rVert\leq R}(\mathrm{d}G)$.
In relation to (\ref{3}) it was shown in  \cite{DRS93} that
\begin{equation}\label{CE}
{\rm vol} \, ({\rm SL}_N(\mathbb R)/ {\rm SL}_N(\mathbb Z)) = \zeta(2) \zeta(3) \cdots \zeta(N),
\end{equation}
where $\zeta(s)$ denotes the Riemann zeta function (see also \cite{Ga14}). 
A result of Siegel \cite{Si36} gives that for a certain non-arithmetic constant $A$, depending
on the normalisation of the measure,
\begin{equation}\label{CZ}
\mathrm{vol}\,   (\mathrm{SL}_N(\mathbb C) /  \mathrm{SL}_N(\mathbb \Z[i]) = A
\zeta_{\Z[i]}(2)\,\zeta_{\Z[i]}(3)\,\cdots\,\zeta_{\Z[i]}(N).
\end{equation}
Here $\zeta_{\Z[i]}(s)$ denotes the Dedekind zeta function for the Gaussian integers,
\begin{equation}\label{5}
\zeta_{\Z[i]}(s)=\frac{1}{4}\sum_{(m,n)\neq(0,0)}\frac{1}{(m^2+n^2)^s}= \zeta(s) \sum_{n=1}^\infty {(-1)^{n-1} \over (2n-1)^s},
\end{equation}
where the second equality is a well known factorisation; see e.g.~\cite{BGMWZ13}.
For future reference we note that for $s=2$ this gives
\begin{equation}\label{CO}
\zeta_{\Z[i]}(2) = \zeta(2) \sum_{n=1}^\infty {(-1)^{n-1} \over (2n-1)^2} = {\pi^2 \over 6} C,
\end{equation}
where
$$
C = \sum_{n=1}^\infty {(-1)^{n-1} \over (2n - 1)^2}
$$
denotes Catalan's constant. In Remark \ref{RF1} below, we will show that in the same
normalisation as used to compute the integral over $G \in {\rm SL}_N(\mathbb C)$, for
$N=2$ (\ref{CZ}) holds with $A=1$.

\section{The Lagrange-Gauss algorithm -- the real case}\label{s3.2}
Our study of lattice reduction in $\C^2$ and $\mathbb H^2$ draws heavily on the theory of lattice reduction in $\R^2$.
For the logical development of our work we must revise some essential aspects of the latter, presenting
in particular theory associated with the Lagrange-Gauss algorithm.
\subsection{Vector recurrence and shortest reduced basis}
Let $\mathcal{B}=\{\mathbf{b}_1,\mathbf{b}_0\}$ with $\lVert\mathbf{b}_1\rVert\leq\lVert\mathbf{b}_0\rVert$ say, be a basis for $\R^2$, and let $\mathcal{L}=\{n_1\mathbf{b}_1+n_0\mathbf{b}_0\,|\,n_1,n_0\in\Z\}$ be the corresponding lattice. The lattice reduction problem in $\R^2$ is to find the shortest nonzero vector in $\mathcal{L}$ (call this $\bm{\alpha}$), and the shortest nonzero vector linearly independent from $\bm{\alpha}$ (call this $\bm{\beta}$) to obtain a new, reduced basis.

Let us suppose that a fundamental cell in $\mathcal{L}$ has unit volume. Then with $\bm{\alpha},\bm{\beta}$ written as column vectors, the matrix $B=[\mathbf{b}_1\, \mathbf{b}_0]$ has unit modulus for its determinant, which we denote $B\in\mathrm{SL}_2^{\pm}(\R)$. Similarly with $V=[\bm{\alpha\,\beta}]$ we have $V\in\mathrm{SL}_2^{\pm}(\R)$. The matrices $B$ and $V$ are related by
\begin{equation}\label{30}
V=BM,\quad M\in\mathrm{SL}_2^{\pm}(\Z).
\end{equation}

The Lagrange-Gauss algorithm finds a sequence of matrices $M_i\in\mathrm{SL}_2^{-}(\Z)$ ($i=1,\ldots,r^{*}$) such that
\begin{equation}\label{31}
M=M_1M_2\cdots M_{r^{*}},\quad M_i=\begin{bmatrix}-m_i&1\\1&0\end{bmatrix}\,(m_i\in\Z)
\end{equation}
(in fact for $B$ chosen with invariant measure, $M$ samples from $\mathrm{SL}_2^{\pm}(\Z)$,
with a restriction on the size of the matrix norm, uniformly;
see \cite{Ri16}).
Defining
\begin{equation}\label{32}
B_{j+1}=B_j\begin{bmatrix}-m_j&1\\1&0\end{bmatrix},\quad B_1=B=[\mathbf{b}_1\, \mathbf{b}_0],
\end{equation}
the first column of $B_j$ is the second column of $B_{j+1}$ so that we can now set
\begin{equation*}
B_j=[\mathbf{b}_j\, \mathbf{b}_{j-1}]
\end{equation*}
for some $2\times1$ columns vectors $\mathbf{b}_j, \mathbf{b}_{j-1}$. Then \eqref{32} reduces to a single vector recurrence
\begin{equation}\label{33}
\mathbf{b}_{j+1}=\mathbf{b}_{j-1}-m_j\mathbf{b}_j.
\end{equation}

The integer $m_j$ in \eqref{33} is chosen to minimise $\lVert\mathbf{b}_{j+1}\rVert$ and is given by
\begin{equation}\label{34}
m_j=\left\lceil\frac{\mathbf{b}_j\cdot \mathbf{b}_{j-1}}{\lVert\mathbf{b}_j\rVert^2}\right\rfloor,
\end{equation}
where $\lceil\,\cdot\,\rfloor$ denotes the closest integer function (boundary case $\lceil\half\rfloor=0$), and so
\begin{equation}\label{35}
\mathbf{b}_{j+1}=\mathbf{b}_{j-1}-\left\lceil\frac{\mathbf{b}_j\cdot \mathbf{b}_{j-1}}{\lVert\mathbf{b}_j\rVert^2}\right\rfloor\mathbf{b}_j.
\end{equation}
Geometrically, the RHS of \eqref{35} is recognised as the formula for the component of $\mathbf{b}_{j-1}$ near orthogonal to $\mathbf{b}_j$. The qualification "near" is required because $m_j$ is constrained to be an integer so that $\mathbf{b}_{j+1}\in\mathcal{L}$.

A basic property of \eqref{33} is that successive vectors are smaller in magnitude whenever $m_{j+1}\neq0$; see e.g.~\cite{Br12}.

\begin{lemma}\label{lemma7}
Suppose $m_{j+1}\neq0$. We have
\begin{equation}\label{36}
\lVert\mathbf{b}_{j+1}\rVert<\lVert\mathbf{b}_j\rVert.
\end{equation}
\end{lemma}
\begin{proof}
Generally
\begin{equation*}
\lceil x\rfloor=x+\epsilon,\quad-\half\leq\epsilon<\half,
\end{equation*}
and so
\begin{equation}\label{37}
\left\lceil x-\lceil x\rfloor\right\rfloor=0.
\end{equation}
Now, taking the dot product of both sides of \eqref{35} with the vector $\mathbf{b}_j$ and dividing both sides by $\lVert\mathbf{b}_j\rVert^2$, use of \eqref{37} with $x=\mathbf{b}_j\cdot \mathbf{b}_{j-1}/\lVert\mathbf{b}_j\rVert^2$ implies
\begin{equation}\label{38}
\left\lceil\frac{\mathbf{b}_j\cdot \mathbf b_{j+1}}{\lVert\mathbf{b}_j\rVert^2}\right\rfloor=0.
\end{equation}
Comparing the LHS of \eqref{38} with the definition of $m_{j+1}$ as implied by \eqref{34} 
upon writing
\begin{equation}\label{38a}
\left\lceil\frac{\mathbf{b}_j\cdot \mathbf b_{j+1}}{\lVert\mathbf{b}_j\rVert^2}\right\rfloor
= \left\lceil
{\lVert\mathbf{b}_{j+1}\rVert^2 \over \lVert\mathbf{b}_{j}\rVert^2}
\frac{\mathbf{b}_j\cdot \mathbf b_{j+1}}{\lVert\mathbf{b}_{j+1}\rVert^2}\right\rfloor,
\end{equation}
we conclude that if $m_{j+1}\neq0$ then \eqref{36} holds, as required.
\end{proof}

Since the vectors in $\mathcal{L}$ with length less than some value $R$ form a finite set,  Lemma 
\ref{lemma7} implies that for some $j=r$ we must have $m_r=0$. Then \eqref{33} gives $\mathbf{b}_{r+1}=\mathbf{b}_{r-1}$. If at this stage $\lVert\mathbf{b}_{r}\rVert\geq\lVert\mathbf{b}_{r-1}\rVert$, the algorithm stops with $r^{*}=r-1$ in \eqref{31}, and outputs
\begin{equation}\label{39}
\bm{\alpha}=\mathbf{b}_{r-1},\quad\bm{\beta}=\mathbf{b}_{r}
\end{equation}
as the reduced basis. If instead $\lVert\mathbf{b}_{r}\rVert <\lVert\mathbf{b}_{r-1}\rVert (= \lVert\mathbf{b}_{r+1}\rVert)$ the algorithm stops with $r^{*}=r$ in \eqref{31} and outputs
\begin{equation}\label{40}
\bm{\alpha}=\mathbf{b}_{r},\quad\bm{\beta}=\mathbf{b}_{r+1}
\end{equation}
as the reduced basis. Equivalently, the recurrence (\ref{31}) is iterated until for some $j = r^*$,
$\lVert\bm{b}_{r^*+1}\rVert \ge \lVert\bm{b}_{r^*}\rVert $, and the output is the 
basis $\bm{\alpha}=\mathbf{b}_{r^*}$ and $\bm{\beta}=\mathbf{b}_{r^*+1}$.

For both \eqref{39} and \eqref{40} it follows from \eqref{38} with $j=r,r-1$ respectively, and the relative length of $\mathbf{b}_{r+1},\mathbf{b}_r$ that
\begin{equation*}
\lVert\bm{\alpha}\rVert\leq\lVert\bm{\beta}\rVert,\quad\left\lceil\frac{\bm{\alpha}\cdot\bm{\beta}}{\lVert\bm{\alpha}\rVert^2}\right\rfloor=0
\end{equation*}
or equivalently
\begin{equation}\label{41}
\lVert\bm{\alpha}\rVert\leq\lVert\bm{\beta}\rVert,\quad\left|\frac{\bm{\alpha}\cdot\bm{\beta}}{\lVert\bm{\alpha}\rVert^2}\right|\leq\half.
\end{equation}
One observes that the final inequality is equivalent to requiring that
\begin{equation}\label{39a}
\lVert\bm{\beta}+n\bm{\alpha}\rVert\geq\lVert\bm{\beta}\rVert, \quad \forall n\in\Z.
\end{equation}
An alternative way to see \eqref{39a} is to recall that the integer value $m_j$ which minimises \eqref{33} is given by \eqref{34}, and to apply this with $\mathbf{b}_{j-1}=\bm{\beta},\,\mathbf{b}_j=\bm{\alpha}$, for which $m_j=0$. Basis vectors which satisfy \eqref{39a}, together with the first inequality in \eqref{41}, are said to be greedy reduced in two dimensions \cite{NS04}. Of fundamental importance is the classical fact that a greedy reduced basis in two dimensions is a shortest reduced basis (the converse is immediate).
\begin{proposition}\label{prop8}
Let $\{\bm{\alpha},\bm{\beta}\}$ be a greedy reduced basis. Then $\{\bm{\alpha},\bm{\beta}\}$ is a shortest reduced basis.
\end{proposition}
\begin{proof}
We follow the proof given in \cite{Ga12}, which begins with the greedy reduced basis inequalities
\begin{equation}\label{39b}
\lVert\bm{\beta}+m\bm{\alpha}\rVert\geq\lVert\bm{\beta}\rVert\geq\lVert\bm{\alpha}\rVert, \quad \forall m\in\Z.
\end{equation}
Let $\mathbf{v}=n_1\bm{\alpha}+n_2\bm{\beta}$ be any nonzero element of $\mathcal{L}$. In the case $n_2=0$ we have that $\mathbf{v}$ and $\bm{\alpha}$ are linearly dependent and
 it is immediate that $\lVert\mathbf{v} \rVert \geq\lVert\bm{\alpha}\rVert$. In the case $n_2\neq 0$, write $n_1=qn_2+r$ with $q,r\in\Z$ such that
\begin{equation}\label{39c}
0\leq r<|n_2|.
\end{equation}
Then
\begin{equation*}
\mathbf{v}=r\bm{\alpha}+n_2(\bm{\beta}+q\bm{\alpha})
\end{equation*}
and thus by the triangle inequality
\begin{align}\label{39d}
\lVert\mathbf{v}\rVert &\geq |n_2|\lVert\bm{\beta}+q\bm{\alpha}\rVert-r\lVert\bm{\alpha}\Vert \nonumber \\
&= (|n_2| - r) \lVert\bm{\beta}+q\bm{\alpha}\rVert + r (\lVert\bm{\beta}+q\bm{\alpha}\rVert - \lVert \bm{\alpha} \rVert ).
\end{align}
Now by \eqref{39b}, $\lVert\bm{\beta}+q\bm{\alpha}\rVert-\lVert\bm{\alpha}\rVert\geq 0$ and so
\begin{equation}\label{39e}
\lVert\mathbf{v}\rVert\geq\left(|n_2|-r\right)\lVert\bm{\beta}+q\bm{\alpha}\rVert\geq\lVert\bm{\beta}+q\bm{\alpha}\rVert,
\end{equation}
where the second inequality follows from \eqref{39c}. Finally, applying \eqref{39b} again gives $\lVert\mathbf{v}\rVert\geq\lVert\bm{\beta}\rVert\geq\lVert\bm{\alpha}\rVert$ as required.
\end{proof}

\subsection{Complex scalar recurrence}

The vector equation \eqref{35} can also be written in scalar form, albeit involving complex numbers \cite{DFV97}. Thus, set $\mathbf{b}_j=(x_j,y_j)$ and write $b_j=x_j+\mathrm{i}y_j$. The fact that
\begin{equation}\label{42}
\frac{b_{j-1}}{b_j}=\frac{\mathbf{b}_j\cdot \mathbf{b}_{j-1}}{\lVert b_j\rVert^2}+\mathrm{i}\frac{\mathrm{det}\,B}{\lVert b_j\rVert^2},\quad B=[\mathbf{b}_{j}\,\mathbf{b}_{j-1}]
\end{equation}
then allows \eqref{35} to be written
\begin{equation*}
b_{j+1}=b_{j-1}-\left\lceil\mathrm{Re}\,\frac{b_{j-1}}{b_j}\right\rfloor b_j,
\end{equation*}
or equivalently, with $z_j=b_j/b_{j-1}$ ($b_{j-1}\neq0$)
\begin{equation}\label{43}
z_{j+1}=\frac{1}{z_j}-\left\lceil\mathrm{Re}\,\frac{1}{z_j}\right\rfloor.
\end{equation}
With $\alpha$ and $\beta$ the complex numbers corresponding to the vectors $\bm{\alpha}$ and $\bm{\beta}$, setting $z=\beta/\alpha$ the conditions \eqref{41} for a reduced basis read
\begin{equation}\label{44}
|z|\geq1,\quad|\mathrm{Re}\,z|\leq\half.
\end{equation}
The inequalities \eqref{44} are recognised as specifying the fundamental domain in the upper half plane model of hyperbolic geometry, up to details on the boundary; see e.g. \cite{Te13}. Starting with $r_1=b_1/b_0$, $|r_1|<1$, the recurrence \eqref{43} is to be iterated until $|r_{j+1}|\geq1$.

As already noted in \cite{Fo16}, the Haar measure for $\mathrm{SL}_N(\R)$ with $N=2$ can be parametrised in terms of variables which allow for a seemingly different simplification of the inequalities \eqref{41}, which can in fact be identified with \eqref{44}. The variables of interest come about by writing $V\in\mathrm{SL}_2(\R)$ in the form $V=QR$, where $Q$ is a real orthogonal matrix with determinant $+1$ and $R$ is an upper triangular matrix with positive diagonal entries,
\begin{equation}\label{45}
R=\begin{bmatrix}r_{11}&r_{12}\\0&r_{22}\end{bmatrix},\quad r_{22}=1/r_{11}.
\end{equation}

With $V=[\bm{\alpha\,\beta}]$, the matrix $Q$ can be used to rotate the lattice so that $\bm{\alpha}$ lies along the positive $x$-axis. Thus \eqref{45} gives $\bm{\alpha}=(r_{11},0)$, $\bm{\beta}=(r_{12},1/r_{11})$ and the inequalities \eqref{41} read
\begin{equation}\label{46}
r_{12}^2+r_{22}^2\geq r_{11}^2,\quad2|r_{12}|\leq r_{11}.
\end{equation}
Further, \cite[Eq. (4.13)]{Fo16} tells us that the invariant measure, restricted to the fundamental domain of the shortest basis vectors, 
in the coordinates $r_{11}$ and $r_{12}$ is equal to 
\begin{equation}\label{47}
2\pi\chi_{r_{11}/2\geq|r_{12}|\geq A_{r_{11}}(r_{11}^2-1/r_{11}^2)^{1/2}}\,\mathrm{d}r_{11}\mathrm{d}r_{12},
\end{equation}
where $A_r=1$ for $r\geq1$, $A_r=0$ otherwise. In relation to \eqref{43} and \eqref{44}, we should introduce the scaled vector $\frac{1}{|\bm{\alpha}|}\bm{\beta}=(r_{12}/r_{11},1/r_{11}^2)$ and thus identify $z=r_{12}/r_{11}+\mathrm{i}/r_{11}^2$. The inequalities \eqref{46} then reduce to \eqref{44}, while changing variables in the invariant measure \eqref{47} gives
\begin{equation}\label{48}
\pi\chi_{x^2+y^2>1}\chi_{|x|<1/2}\chi_{y>0}\frac{\mathrm{d}x\mathrm{d}y}{y^2}.
\end{equation} 
The factor $\mathrm{d}x\mathrm{d}y/y^2$, in keeping with the remark below \eqref{44}, is familiar as the invariant measure in the upper half plane model of hyperbolic geometry \cite{Te13}.

Distributions for the lengths of $|| \bm{\alpha} ||$ and $|| \bm{\beta} ||$ can be computed by appropriate integrations
over (\ref{47}) and (\ref{48}) \cite{Fo16}. In the present context, the first calculation of this type appears to have
been carried out by Shlosman and Tsfasman \cite{ST01}, who computed the distribution of the random variable
$\pi/(4 y) = \pi r_{11}^2/4$ --- this has the interpretation as the sphere (disk) packing density.
Integrations with respect to (\ref{47})  are also a feature of exact calculations for the distribution of certain scaled
diameters for random $2k$-regular circulant graphs with $k=2$ \cite{MS13}; of the study of kinetic transport in
the two-dimensional periodic Lorenz gas \cite{MS08}; and of calculations relating to the asymptotics of
certain random linear congruences $\mod p$, as $p \to \infty$ \cite{SV05}, amongst other recent examples.

\section{Lattice reduction in $\C^2$}\label{s4}
\subsection{The complex Lagrange-Gauss algorithm}\label{s4.1}
We seek a generalisation of the Lagrange-Gauss lattice reduction algorithm to the case of lattices in $\C^2$. As a first task, an appropriate generalisation of the integers in the complex plane must be identified. As well as closure under addition and multiplication, inspection of the proof of Proposition 
\ref{prop8} 
tells us that these complex integers should permit a Euclidean algorithm with the absolute value function as norm. This requirement permits the choices
\begin{align}
\Z[\sqrt{D}]&=\{n_1+n_2\sqrt{D}:n_1,n_2\in\Z\}, \quad D=-1,-2 \label{Z1}
\\ \Z\left[\frac{1}{2}(1+\sqrt{D})\right]&=\left\{n_1+\frac{n_2}{2}(1+\sqrt{D}):n_1,n_2\in\Z\right\}, \quad D=-3,-7,-11, \label{Z2}
\end{align}
these being the complex quadratic integers with the desired property \cite{HW79}. 
They have been identified in the context of lattice reduction in the earlier work \cite{Na96}.
The case $D=-1$ gives the Gaussian integers, and $D=-3$ the Eisenstein integers. These two cases have been discussed in the context of complex generalisations of the Lagrange-Gauss algorithm in \cite{YW02,SYHS13}. With the complex integers chosen as in \eqref{Z1} or \eqref{Z2}, and $B=\{\mathbf{b}_0,\mathbf{b}_1\}$ a basis in $\C^2$ such that $|\det[\mathbf{b}_0,\mathbf{b}_1]|=1$ --- this requirement restricting the fundamental unit cell to have unit generalised area, analogous to \eqref{1} the corresponding lattice is defined as
\begin{equation}\label{Z3}
\mathcal{L}=\{m_0\mathbf{b}_0+m_1\mathbf{b}_1\,|\,m_0,m_1\in\Z[w]\}.
\end{equation}
The set $\Z[w]$ with $w$ as in \eqref{Z1} and \eqref{Z2} forms a lattice in $\C$. Around each lattice point $l\in\C$ is its Voronoi region, consisting of all points in $\C$ closer to $l$ than to the other lattice points. A lattice quantizer $D_{\Z[w]}$ maps a given point $z\in\C$ to a closest lattice point (the latter is unique provided $z$ is not on the boundary of the Voronoi region)
\begin{equation}\label{Z4}
D_{\Z[w]}(z)=\mathop{\mathrm{argmin}}_{\lambda\in\Z[w]}\,\lVert\lambda-z\rVert.
\end{equation}

The lattice corresponding to \eqref{Z1} is square for $D=-1$ and rectangular for $D=-2$. The Voronoi region is correspondingly square and rectangular. Because of this
\begin{equation}\label{D1a}
D_{\Z[i]}(z)=\lceil\operatorname{Re}\,z\rfloor+\mathrm{i}\lceil\operatorname{Im}\,z\rfloor
\end{equation}
and
\begin{equation}\label{D1b}
D_{\Z[\sqrt{2}i]}(z)=\lceil\operatorname{Re}\,z\rfloor+\mathrm{i}\sqrt{2}\left\lceil\operatorname{Im}\,z/\sqrt{2}\right\rfloor.
\end{equation}
The lattices corresponding to \eqref{Z2} consist of the disjoint union of two rectangular lattices
\begin{multline*}
\Z\left[\frac{1}{2}(1+\sqrt{D})\right]=\{n_1+n_2\sqrt{D}\,:\,n_1,n_2\in\Z\}
\\ \cup\{(n_1+1/2)+(n_2+1/2)\sqrt{D}\,:\,n_1,n_2\in\Z\}.
\end{multline*}
Denoting these $\mathcal{L}_1,\mathcal{L}_2$ respectively we have
\begin{align*}
D_{\mathcal{L}_1}(z)&=\lceil\operatorname{Re}\,z\rfloor+\sqrt{D}\lceil\operatorname{Im}\,z/\sqrt{-D}\rfloor
\\ D_{\mathcal{L}_2}(z)&=\lceil\operatorname{Re}(z-1/2)\rfloor+\sqrt{D}\left\lceil\operatorname{Im}\left(z-\frac{\sqrt{D}}{2}\right)/\sqrt{-D}\right\rfloor + {1 + \sqrt{D} \over 2}
\end{align*}
and so
\begin{equation}\label{Db}
D_{\Z\left[\frac{1}{2}(1+\sqrt{D})\right]}(z)=\mathop{\mathrm{argmin}}_{\beta\in\{D_{\mathcal{L}_1}(z),D_{\mathcal{L}_2}(z)\}}|\beta-z|.
\end{equation}
In the case $D=-3$ -- the Eisenstein integers -- the formula \eqref{Db} can be found in \cite{SYHS13}.

The complex Lagrange-Gauss algorithm proceeds by generalising the working of the real case as presented in Section \ref{s3.2}. The equation \eqref{30} holds with $M\in\mathrm{SL}^{\pm}_2(\Z[w])$ and the matrices $M_i$ in \eqref{31} are now elements of $\mathrm{SL_2^-}(\Z[w])$ with $m_i\in\Z[w]$.  To minimise $\lVert\mathbf{b}_{j+1}\rVert$ in \eqref{33} requires
\begin{equation}\label{32a}
m_j=D_{\Z[w]}\left(\frac{\overline{\mathbf b}_j\cdot \mathbf b_{j-1}}{\lVert\mathbf{b}_j\rVert^2}\right)
\end{equation}
and so the analogue of \eqref{35} reads
\begin{equation}\label{33a}
\mathbf{b}_{j+1}=\mathbf{b}_{j-1}-D_{\Z[w]}\left(\frac{\overline{ \mathbf{b}}_j\cdot \mathbf b_{j-1}}{\lVert\mathbf{b}_j\rVert^2}\right)\mathbf{b}_j.
\end{equation}
Next, we would like to establish the analogue of Lemma \ref{lemma7}.
\begin{lemma}\label{lemma9}
Define $\mathbf{b}_{j+1}$ by \eqref{33a}, and with $m_j$ defined by \eqref{32a}, suppose $m_{j+1}\neq 0$. Then we have the inequality \eqref{36}, $\lVert\mathbf{b}_{j+1}\rVert<\lVert\mathbf{b}_j\rVert$.
\end{lemma}
\begin{proof}
Generally
\begin{equation*}
D_{\Z[w]}(\zeta)=\zeta+r,
\end{equation*}
where $r$ is an element of the Voronoi region of the origin in $\Z[w]$, telling us that
\begin{equation*}
D_{\Z[w]}\left(\zeta-D_{\Z[w]}(\zeta)\right)=0
\end{equation*}
(cf.\eqref{37}). Choosing $\zeta=\mathbf{\overline{b}}_j\cdot\mathbf{b}_{j-1}/\lVert\mathbf{b}_j\rVert^2$, after taking the dot product of both sides of \eqref{33a} with respect to ${\mathbf{\overline{b}}_j}$ it follows that
\begin{equation}\label{36a}
D_{\Z[w]}\left(\frac{\overline{\mathbf{b}}_j\cdot\mathbf{b}_{j+1}}{\lVert\mathbf{b}_j\rVert^2}\right)=0 , \text{ or equivalently } \:
D_{\Z[w]}\left(\frac{\mathbf{b}_j \cdot\overline{\mathbf{b}}_{j+1}}{\lVert\mathbf{b}_j\rVert^2}\right)=0 
\end{equation}
(cf.\eqref{38}). But from \eqref{32a}
\begin{equation}\label{36b}
m_{j+1}=D_{\Z[w]}\left(\frac{\overline{\mathbf{b}}_{j+1}\cdot\mathbf{b}_j}{\lVert\mathbf{b}_{j+1}\rVert^2}\right)=
D_{\Z[w]} \left({\lVert\mathbf{b}_j\rVert^2 \over \lVert\mathbf{b}_{j+1}\rVert^2}
\frac{\overline{\mathbf{b}}_{j+1}\cdot\mathbf{b}_j}{\lVert\mathbf{b}_j\rVert^2}\right)
\end{equation}
(cf.~(\ref{38a})).
Comparing \eqref{36b} and \eqref{36a} we see that if $m_{j+1}\neq 0$, then we must have the stated
inequality.
\end{proof}
The complex Lagrange-Gauss algorithm terminates with outputs \eqref{37} or \eqref{38} depending on the validity of $\lVert\mathbf{b}_{r+1}\rVert\geq\lVert\mathbf{b}_r\rVert$ as in the real case, and the vectors $\bm{\alpha}, \bm{\beta}$ satisfying
\begin{equation}\label{36c}
\lVert\bm{\alpha}\rVert\leq\lVert\bm{\beta}\rVert, \quad D_{\Z[w]}\left(\frac{\bm{\overline{\alpha}}\cdot\bm{\beta}}{\lVert\bm{\alpha}\rVert^2}\right)=0.
\end{equation}
From the complex analogue of the text below \eqref{39a} we see that the second equation is equivalent to
\begin{equation}\label{36d}
\lVert\bm{\beta}+n\bm{\alpha}\rVert\geq\lVert\bm{\beta}\rVert, \quad \forall n\in\Z[w],
\end{equation}
telling us that $\{\bm{\alpha},\bm{\beta}\}$ is a greedy reduced basis, as in the real case.

The assumption that $\Z[w]$ is a Euclidean domain with the absolute value as norm allows to deduce the complex analogue of Proposition \ref{prop8}.
\begin{proposition}\label{p10}
Let $\{\bm{\alpha},\bm{\beta}\}$ be a complex greedy reduced basis, and let $\Z[w]$ be one of \eqref{Z1}, \eqref{Z2}. Then $\{\bm{\alpha},\bm{\beta}\}$ is a shortest reduced basis.
\end{proposition}
\begin{proof}
We follow the proof of Proposition \ref{prop8}, now setting $\mathbf{v}=n_1\bm{\alpha}+n_2\bm{\beta}$, $n_1,n_2\in\Z[w]$. In the case $n_2\neq 0$, the assumption that $\Z[w]$ is a Euclidean domain with the absolute value as norm allows us to write
\begin{equation*}
n_1=qn_2+r, \quad q,r\in\Z[w]
\end{equation*}
with
\begin{equation*}
0\leq |r|<|n_2|.
\end{equation*}
Equations \eqref{39d} and \eqref{39e} again hold, with $r$ replaced by $|r|$, implying $\lVert\mathbf{v}\rVert\geq\lVert\bm{\beta}\rVert\geq\lVert\bm{\alpha}\rVert$ as required.
\end{proof}

\subsection{Quaternion scalar recurrence}\label{s4.2}
We saw how the real vector equation \eqref{33} could also be written in the complex scalar form \eqref{40}. Here we will show how the complex vector equation \eqref{33a} can be written in a quaternion scalar form; for the latter recall the definitions at the beginning of \S \ref{s2.1}.

Writing a pair of complex basis vectors $\mathbf{b}_l=(w_l,z_l)$, $w_l,z_l\in\C$, define
\begin{equation}\label{qwz}
q_l=w_l+{\rm j} z_l, \quad |q_l|^2=|w_l|^2+|z_l|^2,
\end{equation}
where the unit $\mathrm{i}$ in $w_l,z_l$ is to be regarded as part of the quarternion algebra (note that we have
chosen to have the unit $\mathrm{j}$ to the left). With $V$ the $2\times 2$ matrix with complex vectors $\mathbf{b}_{l-1}$ and $\mathbf{b}_l$ as its columns, analogous to \eqref{40} one can check
\begin{equation}\label{40d}
{q_l}^{-1} q_{l-1} =\frac{\mathbf{\overline{b}}_l\cdot\mathbf{b}_{l-1}}{\lVert\mathbf{b}_l\rVert^2}+ 
\mathrm{j} \frac{\det V}{\lVert\mathbf{b}_l\rVert^2}
\end{equation}
(cf.\eqref{43}). 
Another viewpoint on (\ref{40d}) is in terms of the so-called Cayley--Dickson doubling formula. Thus for
$a,b,c,d \in \mathbb C$ define
\begin{equation}\label{CD}
\overline{(a,b)} = (\overline{a}, - b), \qquad
(a,b)(c,d) = (ac - d \overline{b}, \overline{a} d + c b).
\end{equation}
Identify $(a,b) = a + \mathrm{j} b $. Then these rules together with $q_l^{-1} q_{l-1} = |q_l|^{-2} \overline{q}_l q_{l-1}$ and
$\overline{q}_l = (\overline{a}, - b)$, $q_{l-1} = (c,d)$ together with the fact
that complex numbers commute imply (\ref{40d}).

Consequently, the complex vector recurrence \eqref{33a}, rearranged so that order of multiplication in the
last term is reversed (this is in keeping with the unit $\mathrm{j}$ in (\ref{qwz}) being to the left, and thus
purely complex multiplication taking place to the right), can be rewritten as the quaternion scalar
recurrence
\begin{equation}\label{40e}
 q_j^{-1} q_{j+1} =  q_j^{-1} q_{j-1} -D_{\Z[w]}\left((\operatorname{Re}+
 \mathrm{i}\operatorname{Im}_i) q_j^{-1} q_{j-1} \right)
\end{equation}
where $\operatorname{Im}_i$ denotes the (real) coefficient of $\mathrm{i}$. 
Now writing $Q_j^{-1} =  q_j^{-1} q_{j-1}$ gives the analogue of (\ref{43}),
\begin{equation}\label{40f}
Q_{j+1} = {1 \over Q_j} + D_{\Z[w]}\Big ( (\operatorname{Re}+\mathrm{i}\operatorname{Im}_i){1 \over Q_j} \Big ).
\end{equation}

\subsection{The Gram-Schmidt basis for the Gaussian integers}\label{s4.3}
In the real case the inequalities \eqref{41} specifying a shortest reduced basis can also be obtained by transforming the basis vectors to a Gram-Schmidt basis. In the complex case this can be achieved by writing $V=UT$, where $U\in\mathrm{SU}(2)$ and
\begin{equation}\label{50}
T=\begin{bmatrix}t_{11}&t_{12}^{(r)}+\mathrm{i}t_{12}^{(i)}\\0&t_{22}\end{bmatrix}, \quad t_{11}>0,\,t_{22}=1/t_{11}.
\end{equation}
Recalling the text above \eqref{10a}, and making use of the known change of variables from the elements of $V$ to $\{U,T\}$ (see e.g. \cite[Prop.~3.2.5]{Fo10}) the invariant measure \eqref{8} for $N=2$ can be written
\begin{equation}\label{50a}
\left(\frac{1}{2\pi}\right)\delta(1-t_{11}t_{22})t_{11}^3t_{22}
\mathrm{d}t_{11}\mathrm{d}t_{22}
\mathrm{d} t_{12}^{(r)} \mathrm{d} t_{12}^{(i)}
(U^\dagger\mathrm{d}U).
\end{equation}
Also, with $\bm{\alpha}=(t_{11},0),\,\bm{\beta}=(t_{12}^{(r)}+
\mathrm{i}t_{12}^{(i)},1/t_{11})$ the inequalities implied by \eqref{36c} for $w=\mathrm{i} $ read
\begin{equation}\label{50b}
t_{11}^2\leq (t_{12}^{(r)})^2+(t_{12}^{(i)})^2+(1/t_{11})^2,\quad 2|t_{12}^{(r)}|\leq t_{11},\quad 2|t_{12}^{(i)}|\leq t_{11}.
\end{equation}
Integrating over $U$ using \eqref{13}, and integrating over $t_{22}$ shows that as a function of the variables $\{t_{11},t_{12}^{(r)},t_{12}^{(i)}\}$ the invariant measure restricted to the domain of the shortest reduced basis is equal to
\begin{equation}\label{50c}
(2\pi^2)t_{11}\chi_{t_{11}^2\leq (t_{12}^{(r)})^2+(t_{12}^{(i)})^2+(1/t_{11})^2}\chi_{2|t_{12}^{(r)}|\leq t_{11}}\chi_{2|t_{12}^{(i)}|\leq t_{11}}\mathrm{d}t_{11}\mathrm{d}t_{12}^{(r)}\mathrm{d}t_{12}^{(i)}.
\end{equation}
Now introduce the scaled vector
\begin{equation*}
\frac{1}{|\bm{\alpha}|}\bm{\beta}=\left((t_{12}^{(r)}+\mathrm{i}t_{12}^{(i)})/t_{11},1/t_{11}^2\right)
\end{equation*}
and set $q=(t_{12}^{(r)}+\mathrm{i}t_{12}^{(i)})/t_{11}+\mathrm{j}/t_{11}^2$. Write
\begin{equation*}
x_1=t_{12}^{(r)}/t_{11},\quad x_2=t_{12}^{(i)}/t_{11},\quad x_3=1/t_{11}^2.
\end{equation*}
In these variables the invariant measure \eqref{50c} reads
\begin{equation}\label{50d}
\pi^2\chi_{x_1^2+x_2^2+x_3^2>1}\chi_{|x_1|\leq\frac{1}{2}}\chi_{|x_2|\leq\frac{1}{2}}\chi_{x_3>0}\frac{\mathrm{d}x_1\mathrm{d}x_2\mathrm{d}x_3}{x_3^3}.
\end{equation}
The factor $\mathrm{d}x_1\mathrm{d}x_2\mathrm{d}x_3/x_3^3$ is recognised as the invariant measure for hyperbolic 3-space.

\subsection{Statistics of the shortest reduced basis for the Gaussian integers}\label{s4.4}
In the case of the Gaussian integers, the statistics of the corresponding shortest basis vectors are determined by appropriate integration over \eqref{50c} --- $t_{11}$ is the length of the shortest vector, $\left((t_{12}^{(r)})^2+(t_{12}^{(i)})^2+(1/t_{11})^2\right)^{1/2}$ is the length of the second shortest vector, while for the complex analogue of the cosine of the angle between $\bm{\alpha}$ and $\bm{\beta}$ we have
\begin{equation}\label{51}
\frac{\overline{\bm{\alpha}}\cdot\bm{\beta}}{\lVert\bm{\alpha}\rVert\lVert\bm{\beta}\rVert}=\frac{t_{12}^{(r)}+\mathrm{i}t_{12}^{(i)}}{\sqrt{(t_{12}^{(r)})^2+(t_{12}^{(i)})^2+(1/t_{11})^2}}
\end{equation}
so these variables should be held fixed when computing the corresponding PDF. Integrating \eqref{50c} over all variables gives the volume of the invariant measure (\ref{50a}) restricted to the domain (\ref{50b})
which occurs in the computation of the PDFs as the normalisation. Our first task is to compute this volume.
\begin{proposition}\label{prop11}
Let the volume associated with \eqref{50c} be denoted $\mathrm{vol}\,\widehat{\Gamma}$. We have
\begin{equation}\label{50e}
\mathrm{vol}\,\widehat{\Gamma}=\frac{2\pi^2}{3}C,
\end{equation}
where $C$ denotes Catalan's constant as defined above \eqref{CO}.
\end{proposition}
\begin{proof}
For notational convenience in \eqref{50c} we  write $t_{11}=t,\,t_{12}^{(r)}=y_1,\,t_{12}^{(i)}=y_2$. Integrating over $y_1$ and $y_2$ gives
\begin{equation}\label{24.1}
(2\pi^2)t\mathrm{d}t\int\chi_{\lVert\mathbf{y}\rVert^2\geq t^2-1/t^2}\chi_{|y_1|\leq t/2}\chi_{|y_2|\leq t/2}\mathrm{d}y_1\mathrm{d}y_2,
\end{equation}
where $\mathbf{y}=(y_1,y_2)$. Geometrically, the integral here corresponds to the area overlap between the outside of a disk of radius $\sqrt{t^2-1/t^2}\,(t\geq 1)$ centred at the origin, and a square of side length $t$ centred at the origin. For $t<1$ the first inequality is always true, and the integral is equal to the area of the square, $t^2$.

It follows that with $V_2(a,b)$ denoting the area of overlap between a disk of radius $a$, and square of side length $2b$, both centred at the origin, \eqref{24.1} can be written
\begin{multline}\label{24.2}
(2\pi^2)\left(t^3\chi_{0<t<1}+\chi_{t>1}t\left(t^2-V_2\left((t^2-1/t^2)^{1/2},t/2\right)\right)\right)\mathrm{d}t
\\ =(2\pi^2)\left(t^3\chi_{0<t<1}+\chi_{t>1}t\left(t^2-(t^2-1/t^2)V_2\left(1,\frac{t}{2(t^2-1/t^2)^{1/2}}\right)\right)\right)\mathrm{d}t.
\end{multline}
An elementary exact calculation gives
\begin{equation}\label{24.3}
V_2(1,a)=\begin{cases}
4a^2, \quad &0<a<1/\sqrt{2}
\\ 4a\sqrt{1-a^2}+4\arcsin a-\pi, \quad &1/\sqrt{2}<a<1
\\ \pi, \quad &1<a
\end{cases}
\end{equation}
(see  \cite{RR97}, for an $n$-dimensional generalisation of this result)
thus reducing \eqref{24.2} to
\begin{multline}\label{24.4}
(2\pi^2)\bigg(\chi_{0<t<1}t^3+\chi_{1<t<(4/3)^{1/4}}t(t^2-\pi(t^2-1/t^2))
\\ +\chi_{(4/3)^{1/4}<t<2^{1/4}}t\left( t^2-(t^2-1/t^2)(4a\sqrt{1-a^2}+4\arcsin a-\pi)\rvert_{a=\frac{t}{2(t^2-1/t^2)^{1/2}}}\right)\bigg)\mathrm{d}t.
\end{multline}
Elementary integration and/or use of computer algebra (we used Mathematica) gives for the integral over $t$
\begin{align}
&\int\chi_{0<t<1}t^3\mathrm{d}t=\frac{1}{4}\label{A1}
\\ &\int\chi_{1<t<(4/3)^{1/4}}t(t^2-\pi(t^2-1/t^2))\mathrm{d}t=\frac{1}{12}\left(1+\pi(-1+\log(64/27))\right)\label{A2}
\\ &\int\chi_{(4/3)^{1/4}<t<2^{1/4}}t\left(t^2-(t^2-1/t^2)(4a\sqrt{1-a^2}-\pi)\rvert_{a=\frac{t}{2(t^2-1/t^2)^{1/2}}}\right)\mathrm{d}t\nonumber
\\ &\phantom{\int\chi_{(4/3)^{1/4}<t<2^{1/4}}}= \frac{1}{12}\left(-4+2\pi-3\pi\log(3/2)-2\sqrt{3}\log(2-\sqrt{3})\right)
\label{A4}
\\ &\int\chi_{(4/3)^{1/4}<t<2^{1/4}}t^3 4\arcsin \frac{t}{2(t^2-1/t^2)^{1/2}}\mathrm{d}t=\frac{1}{12}\left(-\pi+\sqrt{3}\log(7-4\sqrt{3})\right).
\label{A5}
\end{align}
However the remaining integral
\begin{equation*}
\int\chi_{(4/3)^{1/4}<t<2^{1/4}}\frac{4}{t}\arcsin \frac{t}{2(t^2-1/t^2)^{1/2}}\mathrm{d}t
\end{equation*}
does not yield immediately to such an approach. For this integral, to be denoted $J$, we begin with some simple manipulation and the change of variables $1/t^2=s$ to obtain
\begin{equation*}
J=\int_{1/2}^{3/4}\frac{1}{s}\arcsin \frac{1}{2(1-s)^{1/2}}\mathrm{d}s.
\end{equation*}
Computer algebra now gives
\begin{equation}\label{A6}
J=\frac{C}{3}+\frac{\pi}{4}\log\frac{9}{8},
\end{equation}
where $C$ denotes Catalan's constant. Adding \eqref{A1}-\eqref{A6} gives $C/3$. Multiplying by $2\pi^2$ as required by \eqref{24.1} then gives \eqref{50e}.
\end{proof}

\begin{remark}\label{RF1}
In \cite{Fo16} it was shown that the analogue of vol$\,\hat{\Gamma}$ for lattice reduction in $\mathbb R^2$
equals $\pi^2/3$, which is twice the value of vol$\,{\rm SL}_2(\mathbb R)/{\rm SL}_2(\mathbb Z)$
as given by  \eqref{CE} with $N=2$. This can be understood since the space of reduced vectors
contains the involution $\{\bm \alpha, \bm \beta \} \mapsto \{- \bm \alpha, - \bm \beta \}$,
and so maps two-to-one to the fundamental domain of ${\rm SL}_2(\mathbb R)/{\rm SL}_2(\mathbb Z)$.
For the present lattice reduction problem, the mapping from the space of reduced lattice vectors
to ${\rm SL}_2(\mathbb C)/{\rm SL}_2(\mathbb Z[i])$ is four-to-one due to the involutions
$\{\bm \alpha, \bm \beta \} \mapsto \{- \bm \alpha, - \bm \beta \}, \,  \{\mathrm{i} \bm \alpha, -\mathrm{i} \bm \beta \}, \,
 \{-\mathrm{i} \bm \alpha,  \mathrm{i} \bm \beta \}$. Hence
 $
 {\rm vol} \, {\rm SL}_2(\mathbb C)/{\rm SL}_2(\mathbb Z[i]) = {\pi^2 \over 6} C = \zeta_{\mathbb Z[i]}(2),
 $
 where the final equality uses (\ref{CO}). It thus follows from (\ref{RF2}) and (\ref{28a}) that
 \begin{equation}\label{RF3}
 \# \{ \gamma: \, \gamma \in {\rm SL}_2(\mathbb Z[i]), \, || \gamma || \le R \}
 \mathop{\sim}\limits_{R \to \infty} {3 \over \pi C} R^4.
 \end{equation}
 \end{remark}

%

In the proof of Proposition \ref{prop11} the expression \eqref{24.4} corresponds to integrating \eqref{50c} over $t_{12}^{(r)}$ and $t_{12}^{(i)}$, and thus after normalisation by dividing by \eqref{50e} and removal of $\mathrm{d}t$ corresponds to the PDF of the length of the shortest basis vector.
\begin{proposition}\label{prop12}
For random complex lattices in $\C^2$, with the defining basis vectors chosen with invariant measure and spanned using the Gaussian integers, the probability density function for the length of the shortest basis vector is equal to
\begin{multline}\label{4.4}
\frac{3}{C}\bigg\{\chi_{0<t<1}t^3+\chi_{1<t<(4/3)^{1/4}}t(t^2-\pi(t^2-1/t^2))
\\ +\chi_{(4/3)^{1/4}<t<2^{1/4}}\bigg(t^3-t^3\sqrt{3-4/t^4}+\pi(t^3-1/t)
-4(t^3-1/t)\arcsin \frac{t}{2(t^2-1/t^2)^{1/2}}\bigg)\bigg\}.
\end{multline}
\end{proposition}
As noted in the opening paragraph of this section, the length of the second shortest basis vector is given by $r=(y_1^2+y_2^2+1/t^2)^{1/2}$, with $y_1,y_2,t$ as specified above \eqref{24.1}. Changing variables from $t$ to $r$ and imposing the ordering and sign restriction $t/2>y_2>y_1>0$ the functional form in \eqref{50c} transform to
\begin{equation}\label{27.1}
(16\pi^2)\frac{r}{(r^2-y_1^2-y_2^2)^2}\chi_{y_1^2+y_2^2<r^2-1/r^2}\chi_{r^2<y_1^2+y_2^2+1/4y_1^2}\chi_{0<y_1<y_2} \,
{\rm d} r {\rm d} y_1 {\rm d} y_2.
\end{equation}
Integrating over $y_1$ and $y_2$ and normalisation by \eqref{50e} gives the explicit form of the corresponding PDF.
\begin{proposition}
In the setting of Propositions \ref{prop11} and \ref{prop12}, the PDF for the length of the second shortest basis vector is equal to
\begin{align}\label{28.1}
\frac{3}{C}&\bigg\{\chi_{1<r<(4/3)^{1/4}}\pi\frac{r^4-1}{r}+\chi_{(4/3)^{1/4}<r<2^{1/4}}\nonumber
\\ &\times \left(2r\sqrt{3r^4-4}+\frac{4(r^4-1)}{r}\left(\arctan\frac{r^2}{\sqrt{3r^4-4}}+\arctan\frac{r^2\sqrt{3r^4-4}-2r^4+2}{r^4-2}-\frac{\pi}{2}\right)\right)\nonumber
\\ &+\chi_{r>2^{1/4}}\left(2r(r^2-\sqrt{r^4-2})+\frac{4(r^4-1)}{r}\left(\arctan\frac{r^2+\sqrt{r^4-2}}{r^2-\sqrt{r^4-2}}-\frac{\pi}{2}\right)\right)\bigg\}.
\end{align}
\end{proposition}
\begin{proof}
Regarding $r>1$ as a parameter, there are three ranges of $r$ values giving a distinctly shaped region as defined by the three inequalities in \eqref{27.1}, see the figure below.

\begin{figure}[ht!]
\centering
\includegraphics[width=\textwidth]{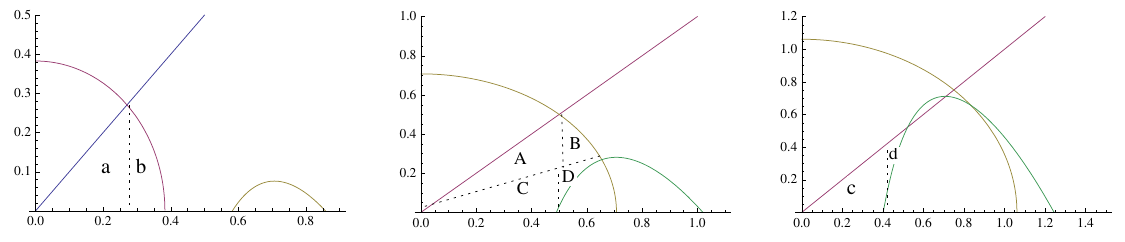}
\caption{This is a plot in the positive $y_1 y_2$-plane of the regions implied by the inequalities in (\ref{27.1}) for fixed values of $r$ in the ranges $<r<(4/3)^{1/4}$,
$(4/3)^{1/4}<r<2^{1/4}$ and $r>2^{1/4}$ respectively.
The common intersection of the inequalities corresponds to the labelled regions in each case.}
\end{figure}

The regions satisfying all the inequalities have been divided into subregions $a,\dots,d$, $A,\dots,D$, which allow for explicit parametrisation of the ranges of integration. Thus for $1<r<(4/3)^{1/4}$,
\begin{equation*}
a=\int_0^{\sqrt{(r^2-r^{-2})/2}}\mathrm{d}y_1 \int_0^{y_1}\mathrm{d}y_2\, ,\quad b=\int_{\sqrt{(r^2-r^{-2})/2}}^{\sqrt{r^2-r^{-2}}}\mathrm{d}y_1\int_0^{\sqrt{r^2-r^{-2}-y_1^2}}\mathrm{d}y_2\,;
\end{equation*}
or $(4/3)^{1/4}<r<2^{1/4}$,
\begin{equation*}
A=\int_0^{\sqrt{(r^2-r^{-2})/2}}\mathrm{d}y_1\int_{(y_1/r^2)\sqrt{3r^4-4}}^{y_1}\mathrm{d}y_2\, ,\quad B=\int_{\sqrt{(r^2-r^{-2})/2}}^{r/2}\mathrm{d}y_1\int_{(y_1/r^2)\sqrt{3r^4-4}}^{\sqrt{r^2-r^{-2}-y_1^2}}\mathrm{d}y_2
\end{equation*}
\begin{equation*}
C=\int_0^{\sqrt{(r^2-\sqrt{r^4-1})/2}}\mathrm{d}y_1\int_0^{(y_1/r^2)\sqrt{3r^4-4}}\mathrm{d}y_2\, ,\quad D=\int_{\sqrt{(r^2-r^{-2})/2}}^{r/2}\mathrm{d}y_1\int_{\sqrt{r^2-y_1^2-1/(4y_1^2)}}^{(y_1/r^2)\sqrt{3r^4-4}}\mathrm{d}y_2\,;
\end{equation*}
and for $r>2^{1/4}$
\begin{equation*}
c=\int_0^{\sqrt{(r^2-\sqrt{r^{4}-1})/2}}\mathrm{d}y_1\int_0^{y_1}\mathrm{d}y_2\, ,\quad d=\int_{\sqrt{(r^2-\sqrt{r^{4}-1})/2}}^{\sqrt{(r^2-\sqrt{r^{4}-1})/4}}\mathrm{d}y_1\int_{\sqrt{r^2-y_1^2-1/(4y_1^2)}}^{y_1}\mathrm{d}y_2.
\end{equation*}
To compute the PDF of the second shortest basis vector, each of these integrations should be extended to include the function $1/(r^2-y_1^2-y_2^2)^2$ for their integrand, as required by \eqref{27.1}. The resulting integrals can all be computed explicitly. Multiplying the result by $16\pi^2r$ as also required by \eqref{27.1}, and normalising by \eqref{50e} we obtain \eqref{28.1}.
\end{proof}

\begin{remark}\label{remark6a} Expanding (\ref{28.1}) for large $r$ one obtains with the help of computer algebra
$$
{3 \over C} \Big ( {1 \over  r^5} + {2 \over 3 r^9} + {\rm O} \Big ( {1 \over r^{13}} \Big ) \Big ).
$$
Multiplying by $dr$ to obtain the corresponding probability measure,
then changing variables $s=1/r$, the resulting PDF thus
has for its leading term in the small $s$ expansion $3 s^3 / C$.
This coincides with the small $t$ behaviour of the PDF for the shortest
vector (\ref{4.4}), and in particular has the same functional dependence on the arithmetic constant $C$.

In the case of lattice reduction applied to bases chosen with invariant measure from SL${}_2(\mathbb R)$,
one can deduce from  \cite[Eq.~(4.16)]{Fo16} that for large $s$ the
PDF for the distribution of the second
shortest basis vector has the large $s$ expansion
$(12/(\pi s)) (1/(2s^2) + 1/(8s^6) + \cdots)$. 
In the variable $\tilde{s} = 1/s$, the leading term in the 
$\tilde{s}  \to 0$ expansion of the transformed PDF is thus $6 \tilde{s}/\pi$.
This is precisely the form of
the PDF of the shortest lattice vector in the range $0 < s < 1$ 
\cite[Eq.~(4.15)]{Fo16}, analogous to what was just exhibited in relation to (\ref{28.1}) and (\ref{4.4}).
Such a property to be expected, as the volume of the unit cell must be unity, and in the case
of one very short vector, and one very long vector, the volume to leading order will just be the product
of the lengths, telling us that such vectors are equal in number.
\end{remark}

The final quantity to be considered is the complex analogue of the cosine of the angle between the shortest reduced basis vectors \eqref{51}. We write
\begin{equation}\label{51a}
\xi_R=\frac{t_{12}^{(r)}}{\sqrt{(t_{12}^{(r)})^2+(t_{12}^{(i)})^2+1/t_{11}^2}},\quad\xi_I=\frac{t_{12}^{(i)}}{\sqrt{(t_{12}^{(r)})^2+(t_{12}^{(i)})^2+1/t_{11}^2}}.
\end{equation}
Their joint distribution can be calculated according to the following result.

\begin{proposition}\label{prop13}
The variables $\xi_R,\xi_I$ specified by \eqref{51a} have joint distribution with PDF equal to
\begin{equation}\label{51b}
-\frac{3}{C}\frac{\log 4\operatorname{max}(|\xi_R|^2,|\xi_I|^2)}{4(1-\xi_R^2-\xi_I^2)^2}
\end{equation}
supported on 
\begin{equation}\label{51c}
\operatorname{max}(|\xi_R|^2,|\xi_I|^2)\leq 1/4.
\end{equation}
\end{proposition}
\begin{proof}
It follows from \eqref{51a} that
\begin{equation*}
t_{12}^{(r)}=\frac{\xi_R}{t_{11}\sqrt{1-\xi_R^2-\xi_I^2}},\quad t_{12}^{(i)}=\frac{\xi_I}{t_{11}\sqrt{1-\xi_R^2-\xi_I^2}}.
\end{equation*}
The Jacobian for the change of variables from $(t_{12}^{(r)},t_{12}^{(i)})=:
(t_{12},s_{12})$
 to $(\xi_R,\xi_I)$ is thus
\begin{equation*}
\left\lvert\det\begin{bmatrix}\frac{\partial t_{12}}{\partial \xi_R}&\frac{\partial t_{12}}{\partial \xi_I}\\ \frac{\partial s_{12}}{\partial \xi_R}&\frac{\partial s_{12}}{\partial \xi_I}\end{bmatrix}\right\rvert=\frac{1}{t_{11}^2(1-\xi_R^2-\xi_I^2)^2}.
\end{equation*}
The functional form in \eqref{50c} thus transforms to
\begin{equation*}
\frac{(2\pi^2)}{t_{11}(1-\xi_R^2-\xi_I^2)^2}\chi_{t_{11}^4<\frac{1}{1-\xi_R^2-\xi_I^2}}\chi_{t_{11}^4>\frac{4\xi_I^2}{1-\xi_R^2-\xi_I^2}}\chi_{t_{11}^4>\frac{4\xi_R^2}{1-\xi_R^2-\xi_I^2}} \,
dt_{11} d\xi_R d \xi_I.
\end{equation*}
Integration over $t_{11}$ in this expression is elementary, and after dividing by the normalisation \eqref{50e} the PDF \eqref{51b} results.
\end{proof}
\begin{corollary}\label{C15}
Let $\xi_R=\xi\cos\theta,\,\xi_I=\xi\sin\theta,\,\xi>0,\,0<\theta<2\pi$ so that $\xi=(\xi_R^2+\xi_I^2)^{1/2}$. The PDF of $\xi$ is equal to
\begin{equation}\label{51d}
-\frac{6 \xi}{C(1-\xi^2)^2}\left( \chi_{0<\xi<1/2} \left(\frac{\pi}{2}\log \xi+C\right)+\chi_{1/2<\xi<1/\sqrt{2}}\int_{\arccos(1/2\xi)}^{\pi/4}\log(4\xi^2\cos^2\theta)\mathrm{d}\theta\right).
\end{equation}
\end{corollary}
\begin{proof}
The Jacobian for the change of variables to polar coordinates is $d\xi_R d \xi_I = \xi d \xi d \theta$.
For $0<\xi<1/2$, the inequality \eqref{51c} is valid for all $0<\theta<2\pi$, and the integral over $\theta$ in \eqref{51b} is equal to
\begin{equation*}
-\frac{3}{4C(1-\xi^2)^2}8\int_0^{\pi/4}\log(4\xi^2\cos^2\theta)\mathrm{d}\theta
\end{equation*}
which after multiplication by $\xi$ evaluates to the first case in \eqref{51d}. For $1/2<\xi<1/\sqrt{2}$, and restricting $\theta$ to the range $0<\theta<\pi/4$, the inequality \eqref{51c} is valid for $\arccos(1/2\xi)<\theta<\pi/4$, and this implies the second case in \eqref{51d}.
\end{proof}

In \cite[Remark 4.5]{Fo16} it was noted that the PDF for the length of the shortest lattice vector in the real
case, which for $0 < s < 1$ was found to equal $6 s/ \pi$, is consistent with a corollary of Siegel's mean value
theorem \cite{Si45} requiring that the expected number of vectors in a disk of radius $R$ be equal to the area of the
disk. Siegel's mean value theorem in \cite{Si45} applies to the case of real lattices, but the more general statement
of the mean value theorem by Weil \cite{We46} (for a clear statement of the latter, see \cite[Th.~3]{Mo96}) removes this requirement, and in particular allows the case of a complex
lattice to be considered. 

 The corollary of the mean value theorem of interest is the requirement that the expected number
of vectors in the punctured complex disk of radius $R$, $\Omega(R)$, be equal to the volume of the disk. The
latter, corresponding to the set $|w|^2 + |z|^2 < R^2$, $w,z, \in \mathbb C$ is equal to the volume of a ball of
radius $R$ in $\mathbb R^4$, which has value ${\pi^2 \over 2} R^4$, so as a consequence of the mean value
theorem we must have
\begin{equation}\label{OR}
\Omega(R) = {\pi^2 \over 2} R^4.
\end{equation}

On the other hand, in light of Propositions \ref{prop12} and \ref{prop13} together, for $R<1$ the punctured complex disk
of radius $R$ will only contain certain Gaussian integer multiplies of the shortest lattice vector $\bm{\alpha}$:
$m \bm{\alpha}$, $m \in \mathbb Z[i]$ with $|m| \, || \bm{\alpha} || < R$, ($m \ne 0$). Define
$|| \bm{\alpha} || / R = s$, and define $N_{\mathbb Z[i]}(p)$ to be the number of Gaussian integers with {\it square}
norm less than or equal to $p$. Use of (\ref{4.4}) for $t < 1$ shows that for $R<1$
\begin{align}\label{100}
\Omega(R) & = {3 \over C} R^4 \sum_{p=1}^\infty N_{\mathbb Z[i]}(p) \int_{(1/(p+1))^{1/2}}^{(1/p)^{1/2}} s^3 \, ds  \nonumber \\
& = {3 R^4 \over 4 C} \sum_{p=1}^\infty N_{\mathbb Z[i]}(p) \Big ( {1 \over p^2}  - {1 \over (p+1)^2} \Big ) 
\: =  {3 R^4 \over 4 C}  \sum_{p=1}^\infty {M_{\mathbb Z[i]}(p) \over p^2},
\end{align}
where $M_{\mathbb Z[i]}(p) :=  N_{\mathbb Z[i]}(p) -  N_{\mathbb Z[i]}(p-1)$ specifies the number of Gaussian integers with square norm equal to $p$. In the notation (\ref{CO}) we have
$$
\sum_{p=1}^\infty {M_{\mathbb Z[i]}(p) \over p^2} = 4 \zeta_{\mathbb Z[i]} (2) = 4 {\pi^2 \over 6} C,
$$
which substituted in (\ref{100}) reclaims (\ref{OR}).

\subsection{The case of Eisenstein integers}\label{s4.5}
For the choices of $w$ as equal to $\frac{1}{2}(1+\sqrt{D})$ for $D=-3,-7,-11$ as in \eqref{Z2} the domain specified by the second condition in \eqref{36c} is a hexagon rather than a square $(D=-1)$, or rectangle $(D=-2)$ in the coordinates $X=t_{12}^{(r)}/t_{11},\, Y=t_{12}^{(i)}/t_{11}$. Specifically, for $D=-3$ the hexagon has vertices at 
\begin{equation}\label{69}
\left(0,\frac{1}{\sqrt{3}}\right), \left(\frac{1}{2},\frac{1}{2\sqrt{3}}\right),
\left(\frac{1}{2},-\frac{1}{2\sqrt{3}}\right),
\left(0,-\frac{1}{\sqrt{3}}\right),
\left(-\frac{1}{2},-\frac{1}{2\sqrt{3}}\right),
\left(-\frac{1}{2},\frac{1}{2\sqrt{3}}\right)
\end{equation}
and is thus a regular hexagon with side length $1/\sqrt{3}$, centred at the origin and with two sides parallel to the $y$-axis.  In terms of inequalities, this hexagon is specified by the requirements that
\begin{equation}\label{4.44a}
|X| < {1 \over 2}, \qquad \sqrt{3} |Y| + |X| < 1.
\end{equation}
Using the variables $\{t_{11},X,Y\}$ the analogue of \eqref{50c} for the invariant measure restricted to the domain of the shortest reduced basis is the expression
\begin{equation}\label{70}
(2\pi^2)t_{11}^3\chi_{1-1/t_{11}^4\leq X^2+Y^2}\chi_{(X,Y)\in\mathcal{H}}\mathrm{d}t_{11}\mathrm{d}X\mathrm{d}Y,
\end{equation}
where $\mathcal{H}$ denotes the above specified regular hexagon.

Analogous to the computation of (\ref{4.4}), the statistics of the shortest reduced basis can be obtained by appropriate integration over \eqref{70}. We begin with the normalisation, obtained by integrating \eqref{70}.
\begin{proposition}\label{p16}
Let the volume associated with \eqref{70} be denoted $\mathrm{vol}\, \widehat{\Gamma}_\mathcal{H}$. We have
\begin{equation}\label{70a1}
\mathrm{vol}\, \widehat{\Gamma}_\mathcal{H}=\frac{\pi}{2}\log 2-\frac{3\pi}{8}\log 3+\frac{3}{2}\left(\operatorname{Im}L_2\left(\frac{3-i\sqrt{3}}{6}\right)+\operatorname{Im}L_2\left(\frac{3+i\sqrt{3}}{4}\right)\right),
\end{equation}
where
\begin{equation}\label{70b}
L_2(z)=\sum_{n=1}^\infty\frac{z^n}{n^2}
\end{equation}
is the dilogarithm function.
\end{proposition}
\begin{proof}
For $t_{11}>1$ the inequalities in \eqref{70} correspond to the overlap between the regular hexagon $\mathcal{H}$ with vertices \eqref{69} and the outside of a circle of radius $1-1/t_{11}^4$. For $0<t_{11}<1$ the first inequality is always valid, and the remaining factor $\chi_{(X,Y)\in\mathcal{H}}$ is the indicator function of the hexagon. Noting that $\mathcal{H}$ has area $\sqrt{3}/2$ shows that integration over $X$ and $Y$ in \eqref{70} gives the function of $t$
\begin{equation}\label{71}
\chi_{0<t<1}t^3\frac{\sqrt{3}}{2}+\chi_{t>1}t^3\left(\frac{\sqrt{3}}{2}-V^{\mathcal{H} d}\left((1-1/t^4)^{1/2}\right)\right),
\end{equation}
where $V^{\mathcal{H} d}(a)$ is the area of overlap between the hexagon $\mathcal{H}$ and a disk of radius $a$ centred at the origin.

Elementary considerations give
\begin{equation}\label{72}
V^{\mathcal{H} d}(a)=\begin{cases}
\pi a^2, \quad &0<a<1/2,
\\ \pi a^2-6a^2\arctan(4a^2-1)^{1/2}+\frac{3}{2}(4a^2-1)^{1/2}, \quad &1/2<a<1/\sqrt{3},
\\ \frac{\sqrt{3}}{2}, \quad &a>1/\sqrt{3}.
\end{cases}
\end{equation}
If we write
\begin{equation*}
\mathrm{vol}\, \widehat{\Gamma}_\mathcal{H}=V_1+V_2,\quad V_2=-6\int_{(4/3)^{1/4}}^{(3/2)^{1/4}}\frac{1}{t}\arctan\left(3-\frac{4}{t^4}\right)^{1/2}\mathrm{d}t
\end{equation*}
then the integral over $t$ specifying $V_1$ as implied by \eqref{71} and \eqref{72} can either be done by elementary computation or the use of computer algebra and gives
\begin{equation}\label{73}
V_1=\frac{\pi}{4}\log\frac{3}{2}.
\end{equation}
For the integral defining $V_2$ straightforward changes of variables give
\begin{align}\label{74}
V_2&=-\frac{3}{2}\int_0^{1/\sqrt{3}}\frac{2s}{3-s^2}\arctan s\,\mathrm{d}s\nonumber
\\ &=\frac{\pi}{4}\log\frac{8}{3}-\frac{3}{2}\int_0^{1/\sqrt{3}}\frac{\log(3-s^2)}{1+s^2}\,\mathrm{d}s\nonumber
\\ &=3\pi\log 2-\frac{5\pi}{8}\log 3+\frac{3}{2}\left(\operatorname{Im}L_2\left(\frac{3-i\sqrt{3}}{6}\right)+\operatorname{Im}L_2\left(\frac{3+i\sqrt{3}}{4}\right)\right),
\end{align}
where the second equality uses integration by parts, and the third computer algebra; in the latter $L_2(z)$ is the dilogarithm function. Adding \eqref{73} and \eqref{74} gives the first line on \eqref{70a1}. 
\end{proof}

The volume  \eqref{70a1}, obtained by direct integration, can be written in a simpler form by adopting instead an indirect approach
using Siegel's mean value theorem.

\begin{proposition}
An alternative evaluation of the volume in Proposition \ref{p16} is
\begin{equation}\label{70a}
\mathrm{vol}\, \widehat{\Gamma}_\mathcal{H} =  \frac{1}{4}\operatorname{Im}L_2\left(\frac{1+i\sqrt{3}}{2}\right).
\end{equation}
\end{proposition}

\begin{proof}
According to (\ref{71}), for $0 < t < 1$ the PDF of the shortest vector is ${1 \over {\rm vol} \, \widehat{\Gamma}_\mathcal{H}} 
\frac{\sqrt{3}}{2} t^3$. Siegel's mean value theorem \cite{Si45}, generalised by Weil \cite{We46} to apply in the
present setting, has the consequence that the expected number of lattice points in a (complex) disk of
radius $R$ is equal to the area of that disk (this assumes a unit normalisation of the volume associated
with the integers; see below).

Repeating the considerations which led to (\ref{100}) we obtain
$$
\Omega(R) = {R^4 \over \mathrm{vol}\, \widehat{\Gamma}_\mathcal{H}} \Big ( {\sqrt{3} \over 2} \Big )
{1 \over 4}
\sum_{(m,n) \in \mathbb Z^2 \atop (m,n) \ne (0,0)}
{1 \over ((m+n/2)^2 + n^2 (3/4))^2}.
$$
As an analytic function in $s$ one has (see e.g.~\cite[Eq.~(1.4.16)]{BGMWZ13})
$$
\sum_{(m,n) \in \mathbb Z^2 \atop (m,n) \ne (0,0)}
{1 \over ((m+n/2)^2 + n^2 (3/4))^s} =  \sum_{(m,n) \in \mathbb Z^2 \atop (m,n) \ne (0,0)}
{1 \over (m^2 + mn + n^2 )^s} = 6 \zeta(s) g(s),
$$
where $\zeta(s)$ denotes the Riemann zeta function and
$$
g(s) = 1 - 2^{-s} + 4^{-s} - 5^{-s} + 7^{-s} - \cdots =
{2 \over \sqrt{3}} {\rm Im} \, {\rm Li}_s(e^{2 \pi i /3}),
$$
Li${}_s$ denoting the polylogarithm function. For $s=2$ (dilogarithm case) the duplication formula
${\rm Li}_2(z^2) = 2 ( {\rm Li}_2(z) + {\rm Li}_2(-z))$ implies $ {\rm Im} \, {\rm Li}_s(e^{2 \pi i /3}) =
{2 \over 3}   {\rm Im} \, {\rm Li}_s(e^{ \pi i /3}) $ and so substituting (\ref{70a}) we see the latter is
valid provided
\begin{equation}\label{SR}
\Omega(R) = \Big ( {4 \over 3} \Big ) \Big ( {\pi^2 R^4 \over 2} \Big ).
\end{equation}

This is a factor ${4 \over 3}$ bigger than (\ref{OR}). To understand this, we note that as a lattice in $\mathbb R^2$,
$\mathbb Z[i]$ has unit cells of area 1, while $\mathbb Z({1 \over 2} + i \sqrt{3})$ has unit cells of area ${\sqrt{3} \over 2}$.
The latter creates a scale factor which when raised to the power of $d$ (the (complex) dimension of the lattice; here $d=2$)
should be included in the meaning of $\Omega(R)$ (for a real lattice, choosing even integers rather than integers best
illustrates this point), thus implying (\ref{SR}).
\end{proof}

The (un-normalised) PDF for the length of the shortest basis vector is given by \eqref{71}. Normalising by \eqref{70a} and substituting \eqref{72} allows us to specify the analogue of Proposition \ref{prop12} in the case of the Eisenstein integers. 
\begin{proposition}
For random complex lattices in $\C^2$, with the defining basis vectors chosen with invariant measure and the lattice formed using the Eisenstein integers, the PDF for the length of the shortest basis vector is equal to
\begin{multline}\label{75}
\frac{1}{\mathrm{vol}\, \widehat{\Gamma}_\mathcal{H}}\bigg\{\chi_{0<t<1}t^3\frac{\sqrt{3}}{2}+\chi_{1<t<(4/3)^{1/4}}t^3\left(\frac{\sqrt{3}}{2}-\pi\left(1-\frac{1}{t^4}\right)\right)
\\ +\chi_{(4/3)^{1/4}<t<(3/2)^{1/4}}t^3\bigg(\frac{\sqrt{3}}{2}-\pi\left(1-\frac{1}{t^4}\right)+6\left(1-\frac{1}{t^4}\right)\arctan\left(3-\frac{4}{t^4}\right)^{1/2}
\\ -\frac{3}{2}\left(3-\frac{4}{t^4}\right)^{1/2}\bigg)\bigg\},
\end{multline}
where $\mathrm{vol}\, \widehat{\Gamma}_\mathcal{H}$ is given by \eqref{70a}.
\end{proposition}

We have not attempted to compute the PDF of the second shortest basis vector, due to the complexity
of the calculation as evident from the proof of Proposition \ref{prop12}. However, the computation
of the joint distribution of
\begin{equation}\label{zXY}
\xi_R = {X \over \sqrt{X^2 + Y^2 + 1/ t_{11}^4}}, \qquad
\xi_I = {Y \over \sqrt{X^2 + Y^2 + 1/ t_{11}^4}}
\end{equation}
and thus the analogue of Proposition \ref{prop13} is a straightforward computation.

\begin{proposition}\label{prop13a}
The joint distribution of the variables $\xi_R$, $\xi_I$ as specified by (\ref{zXY}) has PDF
$$
- {1 \over {\rm Im} \, L_2((1+i \sqrt{3})/2)}
{\log {\rm max} \, (4 |\xi_R|^2, (|\xi_R| + \sqrt{3} | \xi_I|)^2) \over (1 - \xi_R^2 - \xi_I^2)^2}
$$
supported on ${\rm max} \,(4 |\xi_R|^2, (|\xi_R| + \sqrt{3} | \xi_I|^2)) \le 1$.
\end{proposition}

\section{The quaternion Lagrange-Gauss algorithm}\label{s5}
The definition of the quaternion number system was revised at the beginning of Section \ref{s4.2}. The Hurwitz integers $H$ are the quaternions \eqref{aa} with each $a_i$ either all integers, or all half integers. Their distinguishing feature from the obvious Lipschitz integers, defined as the quaternions \eqref{aa} with each $a_i$ an integer, is that they allow for a Euclidean algorithm \cite{CS03}. With $\mathbf{b}_0,\mathbf{b}_1\in\mathbb{H}^2$ we make use of the Hurwitz integers to define the quaternion lattice
\begin{equation}\label{5.0d}
\mathcal{L}_H=\{m_0\mathbf{b}_0+m_1\mathbf{b}_1\,|\,m_0,m_1\in H\}.
\end{equation}
For notational convenience let us rewrite \eqref{aa} as $a=\sum_{j=0}^3a_je_j$, $a_j\in\R$, and denote $\operatorname{Re}q=a_0$, $\operatorname{Im}_{e_j}q=a_j\,(j=1,2,3)$. For $z\in H$ define the lattice quantizer
\begin{equation}\label{5.0}
D_H(z)=\mathop{\mathrm{argmin}}_{\lambda\in H}\lVert \lambda-z\rVert.
\end{equation}
We see that analogous to \eqref{Db}
\begin{equation*}
D_H(z)=\mathop{\mathrm{argmin}}_{\beta\in\{D_{H_1}(z),D_{H_2}(z)\}}|\beta-z|
\end{equation*}
where
\begin{align*}
D_{H_1}(z)&=\left\lceil\operatorname{Re} z\right\rfloor+\sum_{\nu=1}^3e_\nu\left\lceil\operatorname{Im}_{e_\nu} z\right\rfloor
\\ D_{H_2}(z)&=\left\lceil\operatorname{Re}(z-1/2)\right\rfloor+\frac{1}{2}+\sum_{\nu=1}^3e_\nu\left(\left\lceil\operatorname{Im}_{e_\nu}(z-1/2)\right\rfloor+\frac{1}{2}\right).
\end{align*}
The lattice quantizer is relevant to the formulation of a quaternion Lagrange-Gauss algorithm. Thus the reasoning leading to \eqref{33a} tells us that
\begin{equation}\label{QuR}
\mathbf{b}_{j+1}=\mathbf{b}_{j-1}- \mathbf{b}_j \, D_H\left(\frac{\overline{\mathbf b}_j \cdot\mathbf{b}_{j-1}}{\rVert\mathbf{b_j}\lVert^2}\right)
\end{equation}
(note the order of the multiplication in the final term).
We will see below that the analogues of Lemma \ref{lemma7} and Proposition \ref{prop8} remain true. On the other hand, the rewrite of this
quaternion vector equation to a scalar equation using  the doubling of the quaternions to the octonions as implied
by (\ref{CD}) breaks down. This is because to identify the first component of $(\overline{a}, -b)(c,d)$ as
specified by (\ref{CD}) with a dot product requires that $d \overline{b} =  \overline{b}  d$ --- and thus commutivity ---
which is not true in general for quaternions. 

Iteration of (\ref{QuR}) typically gives smaller vectors, as known in the real and complex cases from Lemmas
\ref{lemma7}  and \ref{lemma9}.

\begin{lemma}\label{Lq2}
Define $m_j$ by (\ref{32a}) with $\mathbb Z[w]$ replaced by $H$. Define $\mathbf b_{j+1}$ by (\ref{QuR}) and
suppose the resulting value of $m_{j+1}$ is nonzero. Then
$$
|| \mathbf b_{j+1} || < || \mathbf b_{j} ||.
$$
\end{lemma}

\begin{proof}
The same proof as for Lemma \ref{lemma7} suffices.
\end{proof}

As in the analogous setting for lattice reduction in $\mathbb R^2$ and $\mathbb C^2$, it follows from Lemma
\ref{Lq2} that the quaternion Lagrange-Gauss algorithm terminates, and furthermore that the output vectors
$\bm{\alpha}, \bm{\beta}$ can be chosen to satisfy
\begin{equation}\label{Oc1}
|| \bm{\alpha} || \le || \bm{\beta} ||, \qquad
D_H \Big ( {\overline{\bm{\alpha}} \cdot \bm{\beta} \over || \bm{\alpha} ||^2} \Big ) = 0.
\end{equation}
The second of these conditions is equivalent to requiring that
$$
|| \bm{\beta} + n \bm{\alpha} || \ge || \bm{\beta} ||, \qquad \forall n \in H
$$
(cf.~going from (\ref{36c}) to (\ref{36d})) and thus $\{ \bm{\alpha}, \bm{\beta} \}$ is a greedy basis.
But we know from the proofs of Propositions  \ref{prop8} and \ref{p10} that subject only to the set of integers --- here the
Hurwitz integers H --- being a Euclidean domain with absolute value for norm, the greedy basis $\{ \bm{\alpha}
, \bm{\beta} \}$ is the shortest reduced basis. It has already been remarked that as distinct from
the Lipschitz integers the Hurwitz integers do allow for a Euclidean algorithm, and it furthermore is true that
the absolute value function is the norm. Hence we have a quaternion analogue of 
Propositions \ref{prop8} and \ref{p10}.

\begin{proposition}
Let $\{ \bm{\alpha}, \bm{\beta} \}$ be a greedy basis for the Hurwitz integer quaternion lattice
(\ref{5.0d}). Then $\{ \bm{\alpha}, \bm{\beta} \}$  is a shortest reduced basis.
\end{proposition}

As for the real and complex cases, a convenient parametrisation of the shortest basis is obtained by using
the Gram--Schmidt basis. Thus one decomposes $V = UT$ where $U \in {\rm SL}_2(\mathbb H)$ and
$$
T = \begin{bmatrix} t_{11}e_0 & t_{12}^0e_0 + \sum_{l=1}^3 e_l t_{12}^l \\
0 & t_{22}e_0 \end{bmatrix}, \qquad t_{11} > 0, \quad t_{22} = 1/ t_{11}.
$$
Since in the Gram--Schmidt basis
$$
\bm{\alpha} = (t_{11}, 0), \qquad \bm{\beta} = \Big ( \sum_{l=0}^3 e_l t_{12}^l, 1/ t_{11} \Big ),
$$
the conditions (\ref{Oc1}) characterising the shortest basis give
$$
1 - 1/ t_{11}^4 \le \sum_{l=0}^3 X_l^2, \qquad
D_H\Big ( \sum_{l=0}^3 e_l X_l \Big ) = 0,
$$
where $X_l = t_{12}^l/ t_{11}$.

Also, the Jacobian associated with the change of variables to the Gram--Schmidt basis is $t_{11}^7 t_{22}^3$ (see e.g.~\cite[Ex.~3.2 q.5(i)]{Fo10}).
Thus for $\mathbb F = \mathbb H$ the (normalised) invariant measure (\ref{8}) in the variables
$\{t_{11}, t_{22}, \{ X_l \}_{l=0}^3 \}$ after integrating out over $t_{22}$ reads
\begin{equation}\label{OPO}
{1 \over {\rm vol} \, \Gamma_{4, H}}
\chi_{1-1/t_{11}^4 \le \sum_{l=0}^3 X_l^2}
\chi_{D_H(\sum_{l=0}^3 e_l X_l) = 0} t_{11}^7 dt_{11} \prod_{l=0}^3 dX_l,
\end{equation}
where ${\rm vol} \, \Gamma_{4, H}$ is the normalisation.

The functional form of the PDF for the length $t$ say of the shortest basis vector can be read off from (\ref{OPO}) in
the region $t < 1$.

\begin{proposition}\label{prop20}
Let  ${\rm vol} \, \Gamma_{4, H}$ be as in (\ref{OPO}). For $0 < t < 1$ the PDF for the length of the shortest
basis vector is equal to 
\begin{equation}\label{VPA}
{1 \over {\rm vol} \, \Gamma_{4, H}} {t^7 \over 2}.
\end{equation}
\end{proposition}

\begin{proof} With $t = t_{11}$, for $0 < t < 1$ the first of the two constraints in (\ref{OPO}) ---
and the only one involving $s$, is always valid. Noting that
\begin{equation}\label{VV}
\int \chi_{D_H(\sum_{l=0}^3 e_l X_l) = 0}  \prod_{l=0}^3 dX_l = {\rm vol} \, V,
\end{equation}
where $V$ denotes the Voronoi cell, then noting that $ {\rm vol} \, V$ is equal to the volume of a fundamental
cell for the lattice in $\mathbb R^4$ corresponding to the Hurwitz integers, the task is to calculate this latter
volume. Since the lattice corresponding to the Hurwitz integers can be generated by
$$
\begin{bmatrix}1/2 & 0 & 0 & 0 \\
1/2 & 1 & 0 & 0  \\
1/2 & 0 & 1 & 0 \\
1/2 & 0 & 0 & 1 \end{bmatrix}
$$
we conclude ${\rm vol} \, V = 1/2$, and (\ref{VPA}) follows.
\end{proof}

From the definition of the Hurwitz integers, and the quantizer $D_H$, 
the constraint $D_H(\sum_{l=0}^3 e_l X_l) = 0$ can be characterised by the inequalities
\begin{equation}\label{XsX}
|X_l| < {1 \over 2} \quad (l=0,\dots, 3) \quad {\rm and} \quad
\sum_{l=0}^3 | X_l| < 1.
\end{equation}
We have not succeeded in extending the method of the proof of Propositions \ref{prop11} and \ref{p16} for a direct
calculation of
\begin{equation}\label{29a}
 {\rm vol} \, \Gamma_{4, H} = \int \chi_{2^{1/4} > t_{11} > 0}
 \chi_{1 - 1/t_{11}^4 \le \sum_{l=0}^3 X_l^2}
 \Big ( \prod_{l=0}^3 \chi_{|X_l| \le 1/2} \Big )
 \chi_{\sum_{l=0}^3 | X_l| \le 1} t_{11}^7 dt_{11} \, \prod_{l=0}^3 d X_l,
 \end{equation}
 where in obtaining this integral we have used the fact $ {\rm vol} \, \Gamma_{4, H}$ is the normalisation in (\ref{OPO}), and that $t_{11}$ is positive
 and can be no bigger than $2^{1/4}$. But we can deduce its value, as we now proceed to demonstrate.

 First, we remark that
 the integrand in (\ref{29a}) is even in the $X_l$, and so can be restricted to positive values of these variables provided we multiply by $2^4$. Doing this, the change of variables $X_l = x_l/t_{11}$, $t_{11} = u^{1/4}$ 
shows
 \begin{equation}\label{29b}
 {\rm vol} \, \Gamma_{4, H} = 4 \int \chi_{2  >  u  > 0}
 \chi_{u^{1/2} - u^{-1/2} \le \sum_{l=0}^3 x_l^2}
 \Big ( \prod_{l=0}^3  \chi_{2^{1/4}/2 > x_l > 0} \chi_{x_l \le u^{1/4}/2} \Big )
 \chi_{\sum_{l=0}^3  x_l \le u^{1/4}}  d u \, \prod_{l=0}^3 d x_l.
 \end{equation}
 This is well suited to approximate numerical evaluation by a Monte Carlo rejection method, which with $10^6$ trials gives the estimate
 $0.105$. In fact Siegel's mean value theorem can be used to indirectly compute the exact value.
 
 \begin{proposition}
 The exact value of the normalisation is
 \begin{equation}\label{Af}  
  \Gamma_{4, H} = {7 \zeta(3) \over 80} \approx 0.1051799 \cdots
  \end{equation}
  \end{proposition}
  
  \begin{proof}
 Let $\Omega(R)$ denote the expected number of vectors in the punctured quaternion disk of radius $R$.
 The fact that as a lattice in $\mathbb R^4$, the Hurwitz integers have unit cell of area ${1 \over 2}$ (recall
 the proof of Proposition \ref{prop20}) tells us that the appropriate version of Siegel's mean value theorem as generalised by Weil \cite{We46} is the
 statement that
  \begin{equation}\label{1Ad}
  \Omega(R) = 2^2 {\pi^4 R^8 \over 24},
  \end{equation}
  where $\pi^4 R^8/ 24$ is the volume of the ball of radius $R$ in $\mathbb R^8$. The factor of $2^2$ is due to the area
  of the unit cell corresponding to the Hurwitz integers being $1/2$; recall the discussion below (\ref{SR}).
  
  On the other hand, starting with (\ref{VPA}), the considerations which led to (\ref{100}) give
  $$
  \Omega(R) = {R^8 \over 16 {\rm vol} \, \Gamma_{4,H}} \sum_{\mathbf q \in H  \backslash \{ \mathbf 0 \}}
  {1 \over | \mathbf q |^8 }.
  $$
  With $\zeta(s)$ denoting the Riemann zeta function, results contained in \cite{Zu74} tell us that
  $$
  \sum_{\mathbf q \in H \backslash \{ \mathbf 0 \}}
  {1 \over | \mathbf q |^8 } = 21 \zeta(3) \zeta(4) = {21 \pi^4 \over 90} \zeta(3)
  $$
  and thus
  \begin{equation}\label{Ae}
  \Omega (R) = {7 \pi^4 R^8 \zeta(3) \over 2^5 \cdot 3 \cdot 5   \cdot {\rm vol} \, \Gamma_{4, H}}
  \end{equation}
  Equating with (\ref{1Ad}) gives (\ref{Af}).
  \end{proof}

 \begin{remark}
  For the PDF of the second shortest basis vector in the real and complex cases, it has been demonstrated in
  Remark \ref{remark6a} that the asymptotic form for large length $s$, after the change of variables
  $s \mapsto 1/s$, is precisely the same as the small-$s$ form of the PDF of the shortest basis vector. Here we
  will demonstrate this same property for the quaternion case.
  
  In (\ref{OPO}), with the quantiser rewritten according to (\ref{XsX}),
 and the change of variables $X_l \mapsto t_{11} X_l$, we set 
   and $\mathbf X=(X_0,\dots,X_3)$, and further change
  variables from $t_{11}$ to $s = (|\mathbf X|^2 + 1/ t_{11}^2)^{1/2}$ --- the length of the second shortest basis vector ---
  to deduce that the PDF of the latter is
  \begin{equation}\label{5R}
  {1 \over {\rm vol} \, \Gamma_{4,H}}
  \int \chi_{|\mathbf X|^2 \le s^2 - 1/s^2} \Big ( \prod_{l = 0}^3 \chi_{|\mathbf X|^2 + 1/ 4 X_l^2 \ge s^2}
  \Big ) \chi_{|\mathbf X|^2 + 1/  ( \sum_{l = 0}^3 |X_l|)^2 \ge s^2} 
  {s \over (s^2 - |\mathbf X|^2)^3} \, \prod_{l = 0}^3 d X_l.
  \end{equation}

   \begin{figure}[t]
\centering
\includegraphics[scale=0.45]{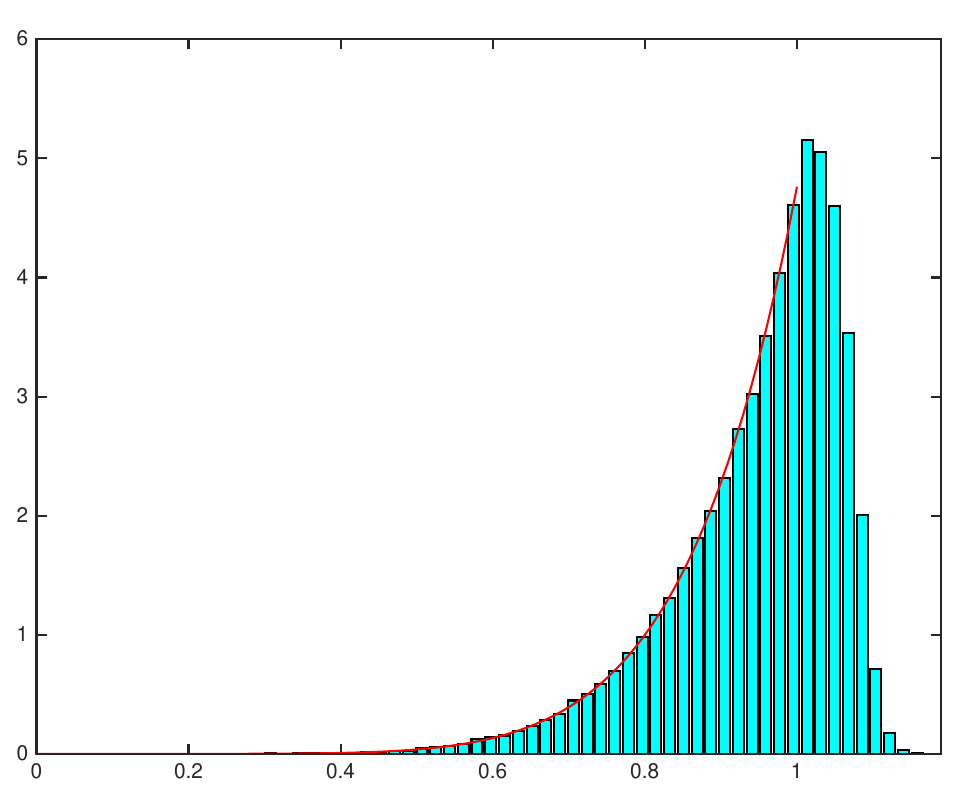}
 \caption{\label{fig1aA}  A total of $10^6$ matrices were sampled from SL${}_2(\mathbb H)$ with invariant measure and
 and bound $R=40$ on the 2-norm. For each, the quaternion Lagrange--Gauss lattice reduction algorithm with respect to the Hurwitz integers
 has been applied to compute the shortest basis vector $\bm \alpha$ . A histogram has been formed for the PDF of 
 $|| \bm \alpha ||$. In the range $0<s<1$ the theoretical prediction
 (\ref{VPA}) with $ \Gamma_{4, H}$ specified by (\ref{Af}) has been superimposed.
}
\end{figure}

  Denote
 {\small  $$
  \Gamma_1 = \{ \mathbf X: \, |\mathbf X|^2 \le s^2 - 1/s^2\}, \:\:
  \Gamma_2 = \cup_{l=0}^3 \{ \mathbf X: \, |\mathbf X|^2 + 1/4 X_l^2 \ge s^2\}, \:\:
  \Gamma_3 = \{ \mathbf X: \, |\mathbf X|^2 + 1/ ( \sum_{l = 0}^3 |X_l|)^2 \ge s^2\},
  $$}
  and for $\mu = 1,2$ let
  {\small $$
  D_\mu = \cup_{l = 0}^3 \{\mathbf X: \, |X_l|^2 \le (s^2 - \sqrt{s^4 - 2^{\mu - 1}})/2^{2\mu-1}\}, \:\:
  R_\mu = \{ \mathbf X: \, ( \sum_{l = 0}^3 |X_l|)^2 \le 4 (s^2 - \sqrt{s^4 - 2^{\mu - 1}})/ 2^{2 \mu -1}\}.
  $$}
  Here $D_1$ ($D_2$) results from replacing $|\mathbf X|^2$  by $|X_l^2$ $(2 X_l^2)$ in $\Gamma_2$, then
  solving for $|X_l|^2$. Similarly, $R_1$ ($R_2$) results from replacing $|\mathbf X|^2$ by
  ${1 \over 2} ( \sum_{l=0}^3 |X_l| )^2$ ($ ( \sum_{l=0}^3 |X_l| )^2$) respectively. By construction
  $$
   D_2  \subseteq \Gamma_2 \subseteq  D_1, \qquad
  R_2  \subseteq \Gamma_3 \subseteq  R_1.
  $$
  Also, as $s \to \infty$,
  $
  \Gamma_2 \subseteq \Gamma_1$ and 
  $$
  D_1,D_2 \to   \cup_{l = 0}^3 \{ X_l: \, 1/(2s) + {\rm O}(1/s^5) \ge   |X_l|\}, \: \:
  R_1, R_2 \to \{\mathbf X: \,1/s + {\rm O}(1/s^5) \ge \sum_{l = 0}^3 |X_l|\}.
  $$
  It follows from the above working that for large $s$ the PDF (\ref{5R}) has the leading asymptotic form
  $$
  {1 \over {\rm vol} \, \Gamma_{4,H}} {1 \over s^5}
  \int  \prod_{l = 0}^3  \,
  \chi_{|X_l| \le 1/2s}  \, \chi_{\sum_{l = 0}^3 |X_l| \le 1/s} \,
  \prod_{l = 0}^3 d X_l.
  $$
  Scaling $s$ from the integral, then recognising what remains as (\ref{VV}) simplifies this to
  $$
  {1 \over 2 {\rm vol} \, \Gamma_{4,H}} {1 \over s^9}.
  $$
 Associating this with a measure and thus multiplying by $ds$, changing variables $s \mapsto 1/s$ we obtain (\ref{VPA}), which was our claim. As discussed in
  Remark \ref{remark6a}, this can be anticipated from the fact that the area of a unit cell is unity.
  \end{remark}
  
  In the quaternion case the analogue of the variables (\ref{51a})  and (\ref{zXY}) are the four variables
  $$
  \xi_l = {X_l \over \sqrt{|\mathbf X|^2 + 1/ t_{11}^4} }\qquad l = 0,\dots, 3.
  $$
  Although we don't give the details, we remark that the joint distribution of these variables can
  be computed to obtain a PDF analogous to those in Propositions \ref{prop13} and \ref{prop13a}.

  Using an extension of the numerical method detailed in \cite{Fo16}  the quaternion version of the Lagrange--Gauss algorithm detailed in
  \S \ref{s5} has been implemented, allowing for the plotting of a histogram approximating the PDF for
  the shortest basis vector. As shown in Figure \ref{fig1aA} this exhibits excellent agreement with the theoretical prediction
  (\ref{VPA})  augmented by (\ref{Af}). We remark that the numerical methods of \cite{Fo16}
  have also been appropriately generalised to provide realisations by way of histograms of the
  PDFs (\ref{4.4}),  (\ref{28.1}), (\ref{51d}) and (\ref{75}). Although we refrain from displaying the results, we remark that again the agreement is excellent.

\section*{Acknowledgements}
This research project is part of the program of study supported by the 
ARC Centre of Excellence for Mathematical \& Statistical Frontiers. We thank F.~Calegari
for the footnote made in relation to (\ref{28a}). We thank too the referee for
a very thorough reading.


\begin{thebibliography}{10}

\bibitem{BGMWZ13}
J.M. Borwein, M.L. Glasser, R.C. McPhedran, J.G. Wan, and I.J. Zucker,
  \emph{Lattice sums then and now}, Cambridge University Press, Cambridge,
  2013.

\bibitem{Br12}
M.R. Bremner, \emph{Lattice basis reduction: an introduction to the {LLL}
  algorithm and its applications}, CRC Press, Boca Raton, FL, 2012.

\bibitem{CS99}
J.H. Conway and N.J.A. Sloane, \emph{Sphere packing, lattices and groups}, 3rd
  ed., Springer-Verlag, Berlin, New York, 1999.
  
 \bibitem{CS03}
J.H. Conway and D.A. Smith, \emph{On quaternions and octonions: their geometry,
arithmetic, and symmetry}, A K Peters Ltd., 2003. 

\bibitem{DFV97}
H.~Daud\'e, P.~Flajolet, and B.~Vall\'ee, \emph{An average-case analysis of the
  Gaussian algorithm for lattice reduction}, Combinatorics, probability and
  computing \textbf{6} (1997), 397--433.

\bibitem{DF17}
P.~Diaconis and P.J. Forrester, \emph{Hurwitz and the origin of random matrix
  theory in mathematics}, Random Matrix Th. Appl. \textbf{6} (2017), 1730001.

\bibitem{DG11}
J.A. D\'iaz-Garc\'ia and R.~Guti\'errex-J\'aimez, \emph{On {W}ishart
  distribution: some extensions}, Linear Alg. Applications \textbf{435} (2011),
  1296--1310.
  
  \bibitem{DRS93}
 W.~Duke, Z.~Rudnick and P.~Sarnak, \emph{Density of integer points on affine
  homogeneous varieties}, Duke Math. J. \textbf{81} (1993), 143--179.


\bibitem{Dy62c}
F.J. Dyson, \emph{The three fold way. {Algebraic} structure of symmetry groups
  and ensembles in quantum mechanics}, J. Math. Phys. \textbf{3} (1962),
  1199--1215.
  
  \bibitem{EB17}  
 D.~El-Baz, 
  \emph{Spherical equidistribution in adelic lattices and applications}, arXiv:1710.07944
  
  \bibitem{EM93}
  A.~Eskin and C.~McMullen, \emph{Mixing, counting, and equidistribution in Lie groups},
  Duke Math. J. \textbf{71} (1993), 181--209.

\bibitem{Fo10}
P.J. Forrester, \emph{Log-gases and random matrices}, Princeton University
  Press, Princeton, NJ, 2010.

\bibitem{Fo16}
\bysame, \emph{Volumes for SL${}_n(\mathbb R)$, the Selberg integral and random
  lattices}, Foundations Comp. Math. DOI 10.1007/s10208-018-9376-1 (2018)

\bibitem{Fo16a}
\bysame, \emph{Octonions in random matrix theory}, Proc.~R.~Soc.~A \textbf{473}
  (2017), 20160800.

\bibitem{FW07p}
P.J. Forrester and S.O. Warnaar, \emph{The importance of the {S}elberg
  integral}, Bull. Am. Math. Soc. \textbf{45} (2008), 489--534.

\bibitem{Ga12}
S.~Galbraith, \emph{Mathematics of public key cryptography}, Cambridge
  University Press, Cambridge, 2012.

\bibitem{GLM09}
Y.H. Gan, C.~Ling, and W.H. Mow, \emph{Complex lattice reduction algorithm for
  low-complexity full-diversity MIMO detection}, IEEE Trans. Signal Proc.
  \textbf{47} (2009), 2701--2710.
  
  \bibitem{Ga14}
  P.~Garrett, \emph{Volume of ${\rm SL}_n(\mathbb Z)\backslash {\rm SL}_n(\mathbb R)$
  and ${\rm Sp}_n(\mathbb Z)\backslash {\rm Sp}_n(\mathbb R)$}, web resource \\
  http://www-users.math.umn.edu/~garrett/m/v/volumes.pdf
  
  \bibitem{GN10}
  A.~Gorodnik and A.~Nevo, \emph{The ergodic theory of lattice subgroups}, vol.~172,
  Annals of Math.~Studies, PUP, Princeton, NJ, 2010

\bibitem{HW79}
G.H. Hardy and E.M. Wright, \emph{An introduction to the theory of numbers},
  5th ed., Oxford science publications, Oxford, 1979.

\bibitem{Hu97}
A.~Hurwitz, \emph{{\"Uber} die {E}rzeugung der {I}nvarianten durch
  {I}ntegration}, Nachr. Ges. Wiss. G\"ottingen (1897), 71--90.

\bibitem{JM59}
H.~Jack and A.M. Macbeath, \emph{The volume of a certain set of matrices},
  Math. Proc. Camb. Phil. Soc. \textbf{55} (1959), 213--223.
  
  \bibitem{MS08}
 J. Marklof and A. Str\"ombergsson,
 \emph{Kinetic transport in the two-dimensional periodic Lorentz gas},
 Nonlinearity \textbf{21} (2008), 1413--1422.
 
  \bibitem{MS13} 
 \bysame,
  \emph{Diameters of random circulant graphs}, Combinatorica 33 (2013) 429--466.
  
 \bibitem{Mo96}   M.Morishita,   \emph{A mean value theorem in adele geometry},
S\"urikaisekikenky\"sho
K\"oky\"uroku (1996),
1--11.
  
 \bibitem{Na96} H. Napias, \emph{A generalization of the LLL-algorithm
 over Euclidean rings or orders}, Journal de Th\'eorie des Nombres
 de Bordeaux \textbf{8} (1996), 387--396.

\bibitem{NS04}
P.Q. Nguyen and D.~Stehl\'e, \emph{Low-dimensional lattice basis reduction
  revisited}, Algorithmic number theory, Lecture notes in computer science,
  vol. 3076, Springer Berlin Heidelberg, 2001, pp.~338--357.
  
  \bibitem{Ri16}
I.~Rivin, \emph{How to pick a random integer matrix? (and other questions)},
  Math. Computation \textbf{85} (2016), 783--797.

\bibitem{RR97}
C.C. Rousseau and O.G. Ruehr, \emph{Problems and solutions. subsection: The
  volume of the intersection of a cube and a ball in $N$-space. two solutions
  by Bernd Tibken and Denis Constales.}, SIAM Review \textbf{39} (1997),
  779--786.

\bibitem{Se44}
A.~Selberg, \emph{Bemerkninger om et multipelt integral}, Norsk. Mat. Tidsskr.
  \textbf{24} (1944), 71--78.

\bibitem{Se01}
I.~Semaev, \emph{A 3-dimensional lattice reduction algorithm}, Proc.~of CALC
  '01 (P.~Huber and M.~Rosenblatt, eds.), Lecture notes in computer science,
  vol. 2146, Springer-Verlag, 2001, pp.~183--193.
  
    
\bibitem{ST01}  S.~Shlosman and M.A.~Tsfasman, \emph{Random lattices and random sphere packings:
typical properties}, Moscow Math. Journal {\bf 1} (2001), 73--89.

\bibitem{Si36}
C.L. Siegel, \emph{The volume of the fundamental domain for some infinite
  groups}, Trans.~Am.~Math.~Soc. \textbf{39} (1936), 209--218.

\bibitem{Si45}
\bysame \emph{A mean value theorem in geometry of numbers}, Ann. Math.
  \textbf{46} (1945), 340--347.
  
   \bibitem{SV05}
A. Str\"ombergsson and A. Venkatesh, \emph{Small solutions to linear congruences and Hecke equidistribution}, Acta Arith., \textbf{118}  (2005), 41--78. 

\bibitem{SYHS13}
Q.T. Sun, J.~Yuan, T.~Huang, and K.W. Shum, \emph{Lattice network codes based
  on {E}isenstein integers}, IEEE Signal Proc. Magazine \textbf{61} (2013),
  2713--2725.

\bibitem{Te13}
A.~Terras, \emph{Harmonic analysis on symmetric spaces --- {E}uclidean space,
  the sphere, and the Poincar\'e upper half plane}, Springer Science+Business
  Media, New York, 2013.

\bibitem{Va11}
S.~Vance, \emph{Improved sphere packing lower bounds from {H}urwitz lattices},
  Advances in Mathematics \textbf{46} (2011), 340--347.


\bibitem{We46}
A.~Weil, \emph{Sur quelques resultats de {S}iegel}, Summa Brasil Math.
  \textbf{1} (1946), 21--39.

\bibitem{WSJM11}
D.~W\"ubben, D.~Seethaler, J.~Jald\'en, and G.~Matz, \emph{Lattice reduction
  algorithm for low-complexity full-diversity mimo detection}, IEEE Signal
  Proc. Magazine \textbf{28} (2009), 70--91.

\bibitem{YW02}
H.~Yao and G.W. Wornell, \emph{Lattice-reduction-aided detectors for mimo
  communication systems}, in Proc. IEEE Global Communications Conf.
  (GLOBE-COM), Taipei, Taiwan (Nov.~2002).
  
  \bibitem{Zu74} I.J.~Zucker, \emph{Lattice sums in $2,4,6$ and 8 dimensions},
  J. Phys. A \textbf{7} (1976), 1568--1575.

\bibitem{ZS01}
K.~Zyczkowski and H.-J. Sommers, \emph{Induced measures in the space of mixed
  quantum states}, J.~Phys. A \textbf{34} (2001), 7111--7125.

\end{thebibliography}

\providecommand{\bysame}{\leavevmode\hbox to3em{\hrulefill}\thinspace}
\providecommand{\MR}{\relax\ifhmode\unskip\space\fi MR }
\providecommand{\MRhref}[2]{%
  \href{http://www.ams.org/mathscinet-getitem?mr=#1}{#2}
}
\providecommand{\href}[2]{#2}

\end{document}